\documentclass[11pt]{amsart}

\usepackage[letterpaper,margin=1in]{geometry}
\usepackage[T1]{fontenc}
\usepackage{lmodern}
\usepackage[expansion=false]{microtype}
\usepackage{amsmath,amssymb,mathtools}
\usepackage{enumitem}
\usepackage{comment}
\IfFileExists{dsfont.sty}{\usepackage{dsfont}}{%
  \let\mathds\mathbf 
}
\usepackage[colorlinks=true,linkcolor=blue!52!black, citecolor=red, backref=page]{hyperref}
\usepackage{color,xcolor}

\allowdisplaybreaks
\newtheorem{theorem}{Theorem}[section]
\newtheorem{proposition}[theorem]{Proposition}
\newtheorem{lemma}[theorem]{Lemma}
\newtheorem{corollary}[theorem]{Corollary}
\newtheorem{question}[theorem]{Question}
\theoremstyle{definition}
\newtheorem{definition}[theorem]{Definition}
\theoremstyle{remark}
\newtheorem{remark}[theorem]{Remark}

\newcommand{\K}{\mathcal K}
\newcommand{\Ped}{\operatorname{Ped}}

\newcommand{\Id}{\operatorname{Id}}
\newcommand{\Cu}{\operatorname{Cu}}
\newcommand{\Tr}{\operatorname{Tr}}
\newcommand{\QT}{\operatorname{QT}}

\title[Maximal Algebraic Ideals]
{Maximal Algebraic Ideals in Nonunital $C^*$-Algebras}

\author{Zhichao Liu}
\address{School of Mathematical Sciences,
Dalian University of Technology,
Dalian, {\rm 116024}, China }
\email{lzc.12@outlook.com}

\author{Xin Ma}
\address{Institute for Advanced Study in Mathematics, Harbin Institute of Technology, Harbin, China, {\rm 150001}}
\email{xma17@hit.edu.cn}

\subjclass[2020]{Primary 46L05, 46L35; Secondary 16D25, 19K14}
\keywords{Maximal ideal,  Cuntz
semigroup, generalized singular number, Pedersen ideal}

\begin{document}

\begin{abstract}
In this paper, we investigate the existence of maximal algebraic two-sided ideals in nonunital $C^*$-algebras. We formulate singular-distribution estimates in terms of lower semicontinuous $2$-quasitraces and introduce a novel criterion, which we call an admissible compact quasitracial triple set. These tools allow us to control membership in algebraic ideals and 
to prove the nonexistence of maximal ideals in nonunital $C^*$-algebras under a closed finite-scale face condition. This nonexistence result applies to a broad class of simple $C^*$-algebras including simple separable stable AF algebras, and nonunital, simple, separable, stably finite, $\mathcal Z$-stable $C^*$-algebras $A$ that admit an approximate identity consisting of increasing projections $(p_n)$ and for which the simplex $\operatorname{QT}_2^1(A,p_1)$ of normalized traces at $p_1$ has finitely many extreme points. 
Moreover, we construct a simple AF algebra whose finite-scale face is dense but not closed and which nevertheless has a maximal dense algebraic ideal. This example yields a negative resolution to a question raised by Ozawa, asking whether every maximal algebraic two-sided ideal in a $C^*$-algebra must be closed.
 \end{abstract}

\maketitle
\tableofcontents

\section{Introduction}

An algebraic ideal of a $C^*$-algebra is a complex-linear two-sided ideal which need not be norm closed. Such ideals are less rigid than norm-closed ideals, but they arise naturally and have been studied since Pedersen's work on measure theory and the work of Pedersen and Petersen on ideals closed under roots \cite{PedersenMeasure,PedersenPetersen}. In particular, every $C^*$-algebra $A$ has a smallest dense ideal, the Pedersen ideal $\Ped(A)$ \cite[Section~5.6]{Pedersen}. In a nonunital algebra this ideal can be proper, even when $A$ is simple, which already shows that algebraic ideals may behave quite differently from their closed counterparts.

Gardella and Thiel proved that every proper algebraic ideal is contained in a prime ideal \cite{GardellaThiel}. Gardella, Kitamura, and Thiel subsequently showed that semiprime ideals are positively spanned, self-adjoint, hereditary, and closed under positive roots \cite{GKT}.  
Throughout the paper, ideals are algebraic complex-linear two-sided
ideals unless specified otherwise.  A \emph{proper ideal} $J$ of $A$
satisfies $0\subsetneq J\subsetneq A$, and a \emph{maximal ideal} is
maximal among these ideals. Recall that a $C^*$-algebra is said to be simple if it has no nonzero proper closed ideals and $A$ is said to be
algebraically simple if $A$ has no nonzero proper algebraic ideals.  It is well-known that $A$ is algebraically simple if and only if $A$ is simple and $\operatorname{Ped}(A)=A$. An obvious example illustrating the difference between these two concepts is the compact operators $\K$, which is simple and its Pedersen ideal consists of the finite-rank operators.

It is known that if all algebraic ideals in $A$ are modular (e.g., $A$ is unital), every maximal algebraic ideal is closed. In addition, it is not hard to see all maximal algebraic ideals in $C_0(X, A)$ are closed provided that $X$ is locally compact Hausdorff space and $A$ is a unital $C^*$-algebra. Ozawa asked whether this remains true for a general $C^*$-algebra \cite{OzawaMO}; see also \cite[Question~2.8]{GardellaThiel} and \cite[Remark~5.15]{GKT}. The question remains open in general. We record the question as follows.

\begin{question}[Ozawa]\label{q:ozawaquestion}
Is every maximal algebraic ideal in a $C^*$-algebra necessarily closed.
\end{question}

For a simple $C^*$-algebra $A$, every nonzero algebraic ideal is dense; hence any maximal ideal in $A$, \emph{if it exists}, must be non-closed. In this case, the mere existence of such an ideal would already answer Ozawa's question in the negative. For non-simple $C^*$-algebras, by contrast, proper closed ideals may exist, yet it remains unknown whether a maximal ideal---closed or not---always exists. The existence question is therefore fundamental: if no such ideal exists, then the issue of whether it is closed becomes vacuous, and Ozawa's question has no object to which it can apply. This naturally leads to the following question, which may be regarded as a preliminary counterpart to Question \ref{q:ozawaquestion}.

\begin{question}\label{q:main-question}
Does every nonunital $C^*$-algebra admit  a maximal ideal?
\end{question}
Zorn's lemma is not directly applicable here: in a nonunital algebra, an increasing chain of proper ideals may exhaust the entire algebra. The principal technical difficulty is therefore to decide membership in maximal ideals. Ordinary Cuntz subequivalence alone is insufficient for this purpose, even when the ideal is maximal. Our main technical contribution in this direction, Theorem \ref{thm:closed-face-factorization}, overcomes this difficulty by controlling membership in nonclosed ideals through the introduction of \textit{generalized singular functions} and a novel criterion called an \textit{admissible compact quasitracial triple set} (Definition~\ref{def:compact-quasitracial-data}). Generalized singular functions are defined via the cutdown distributions $t\mapsto d_\tau((|x|-t)_+)$ associated with a lower semicontinuous $2$-quasitrace $\tau$ (Definition~\ref{eq:nu-definition}); for a faithful trace, this notion coincides with the generalized singular numbers introduced by Fack and Kosaki \cite{FackKosaki} in the associated semifinite von Neumann algebra. Their role is to retain quantitative information about spectral decay and to make this information compatible with the projection-packing process required in Theorem \ref{thm:closed-face-factorization} for detecting membership in semiprime ideals, under the comparison governed by the admissible compact quasitracial triple set. Such a triple set $(T,(p_n),\rho)$ for a $C^*$-algebra $A$ consists primarily of a compact set $T$ of certain lower semicontinuous $2$-quasitraces on $A$, each possibly of infinite or finite scale, an increasing sequence $(p_n)$ of projections, and the radius of comparison $\rho$ with respect to $T$ and $(p_n)$; we denote by $F\subset T$ its finite-scale part.

By combining Theorem \ref{thm:closed-face-factorization} with a slowly decreasing diagonal factorization argument, we rule out the existence of dense maximal algebraic ideals whenever the finite-scale part $F\subset T$ is closed. We note that Proposition \ref{cor:finite-face-pedersen} implies that, in many cases, $F=T$ if and only if $\Ped(A)=A$, i.e., $A$ is algebraically simple. our main theorem below is most naturally applied when $F$
is a proper closed subset of $T$.

\medskip
\noindent\textbf{Main theorem A} (Theorem~\ref{thm:closed-face-dense} and
Corollary~\ref{cor:simple-quasitracial-scale}).\label{thm: main1}
\emph{Suppose A is a $C^*$-algebra admitting a compact quasitracial comparison pair $(T, \rho)$ in the sense of Definition \ref{def:compact-quasitracial-data} and the finite scale part $F\subset T$ is closed. Then $A$ has no dense maximal ideals.  In particular, a simple such $C^*$-algebra
admitting  has no maximal ideals.}
\medskip

This theorem is then applied to establish the nonexistence of maximal ideals in a wide range of stable and nonstable classes. We refer to Theorem \ref{thm:finite-radius-scale} and Corollary \ref{cor:finite-face-tests} for a list of many \(C^*\)-algebras in the literature to which this applies directly. In practice, for a non-unital $C^*$-algebra $A$ with an approximate identity consisting of increasing projections $(p_n)$, the  $T$ is often interpreted as $\QT_2^1(A, [p_1])$, which consists of all lower semicontinuous $2$-quasitraces of $A$, normalized at $p_1$. In this case, $\QT_2^1(A, [p_1])$ is isomorphic to $\QT_2^1(p_1Ap_1)$.

\begin{corollary}[Theorem \ref{thm:stable-finite}]\label{cor1}
   Let $A$ be a stable simple $C^*$-algebra with a increasing projection approximate unit $\{p_n:n\in \mathbb N\}$ with $p_1\neq 0$. Suppose $D=p_1Ap_1$ is stably finite and $\QT_2^1(D)\neq \emptyset$ and $\operatorname{rc}_{\mathrm p}(D)<\infty$. Then $A$ has no maximal ideals. In particular, all simple separable stable AF algebras have no maximal ideals.
\end{corollary}

The next sample result is for $\mathcal Z$-stable $C^*$-algebras.

\begin{corollary}[Corollary \ref{cor:z-stable-projectional}]\label{cor2}
Let $A$ be simple, nonunital, stably finite, and
$\mathcal Z$-stable, with an increasing projection approximate identity
$(p_n)$. If the finite-scale face of $T=\QT_2^1(A,[p_1])$ is closed,
then $A$ has no maximal ideals. In particular, the conclusion holds if
$\partial_eT$ is finite. If $A$ is separable, then the same conclusion holds if $S$ on its finite-scale extreme points $E_{\mathrm fin}$ is uniformly bounded.
\end{corollary}

The following is our second main result, 
We construct a simple AF algebra $A$ in which $F$ fails to be closed in $T$, yet contains a proper non-closed maximal ideal. Consequently, this answers Ozawa's Question \ref{q:ozawaquestion} in the negative.

\medskip
\noindent\textbf{Main theorem B} (Theorem~\ref{thm:af-dense-maximal}).
\emph{There exists a separable simple nonunital AF algebra $A$
admitting a nonzero proper dense maximal algebraic ideal.}
\medskip

We remark that for this  $A$, the finite scale part $F$ is dense and not closed in $T=\QT_2^1(A, [p_1])$ (see Proposition \ref{prop:af-finite-face-density}). Therefore, the closedness of finite scale part $F$ seems essential for the nonexistence of maximal algebraic ideals.

\subsection*{Outline of the paper} 
The paper is organized as follows. Section~\ref{sec:ideals} collects preliminaries on Cuntz subequivalence and algebraic two-sided ideals in $C^*$-algebras. In Section~\ref{sec:singular}, we introduce generalized singular functions for $C^*$-algebras. Section~\ref{sec:closed face} develops the compact quasitracial comparison pair and establishes the key technical result (Theorem~\ref{thm:closed-face-factorization}) and the first main theorem (Theorem \ref{thm:closed-face-dense}). In Section \ref{sec:radiusc}, we mainly collect some results surrounding on radius of comparison for nonunital $C^*$-algebras. In Section~\ref{sec:stable}, we apply this theorem to construct a class of simple examples without maximal ideals. Finally, in Section~\ref{sec:af-dense-maximal}, we construct a separable simple AF algebra admitting a nonzero proper dense maximal algebraic ideal and thus prove the second main theorem (Theorem~\ref{thm:af-dense-maximal}).

\section{Preliminaries}\label{sec:ideals}
\subsection{Cuntz comparison}
\begin{definition}\rm
Let $A$ be a $C^*$-algebra.
Let $p$ and $q$ be two projections in $A$. Recall that $p$ is\textit{ Murray--von Neumann equivalent} to $q$ in $A$, written $p\sim_{\rm MvN} q$, if there exists $x\in A$ such that $x^*x=p$ and $xx^*=q$. We will write $p\preceq q$ if $p$ is equivalent to some subprojection of $q$. If $A$ is not unital, let us denote the minimal unitization of $A$ by $A^\sim$.

\end{definition}


\begin{definition}\rm
  ({\bf Cuntz semigroup}) (see, e.g., \cite{CEI,Cu}) Denote by $A_+$ the positive cone of $A$.  Let $a,b\in A_+$. One says that $a$ is \textit{Cuntz subequivalent} to $b$, denoted by $a\precsim_{\rm Cu} b$, if there exists a sequence $(r_n)$ in $A$ such that $r_n^*br_n\rightarrow a$. One says that $a$ is  \textit{Cuntz equivalent} to $b$, denoted by $a\sim_{\rm Cu} b$, if $a\precsim_{\rm Cu} b$ and $b\precsim_{\rm Cu} a$. The \textit{Cuntz semigroup} of $A$ is defined as ${\rm Cu}(A)=(A\otimes\mathcal{K})_+/\sim_{\rm Cu}$. We will denote the class of $a\in (A\otimes\mathcal{K})_+$ in ${\rm Cu}(A)$ by $[a]$. Note that ${\rm Cu}(A)$ is a positively ordered abelian semigroup with zero (or monoid) when equipped with the addition: $[ a]+[b]=[ a \oplus b]$, and the relation:
  $$
   [ a]\leq [ b] \Leftrightarrow a\precsim b,\quad a,b\in (A\otimes\mathcal{K})_+.
  $$
\end{definition}
Let $x\in A_+$ and $t\in \mathbb{R}_+$, and set $e_t(x)=\max\{x-t,0\}=(x-t)_+$.
\begin{lemma}[\cite{R}]\label{R lem}
Let $A$ be a C*-algebra, let $a,b\in A_+$, and let  $p,q$ be  projections. Then

{\rm (i)} $a \precsim b$ if and only if $(a-\varepsilon)_{+} \precsim b$ for all $\varepsilon>0$;

{\rm (ii)} if $\|a-b\|<\varepsilon$, then $(a-\varepsilon)_{+} \precsim b$;

{\rm (iii)} if $0\leq a\leq b$, then $a\precsim b$;

{\rm (iv)} $p \preceq q$ if and only if $p \precsim q$.
\end{lemma}

\begin{definition}\label{def:hereditary-comparison}
For $a\in A_+$, let
$
 A_a:=\overline{aAa}
$
be the hereditary algebra generated by $a$
and write
$
 p_a=1_{(0,\infty)}(a)\in A^{**}
$
for the support (spectral) projection of $a$.
\end{definition}
The following lemma is well-known.
\begin{lemma}\label{lem:projection-lifting}
Let $p\in A$ be a projection and $a\in A_+$. If $ p\precsim a$, then there exists a partial isometry $v\in A$ such that $p=v^*v$ and $vv^*\in A_a$.
\end{lemma}

\begin{definition}\rm {\rm (}\cite{BK,BH}{\rm )}
 A (\textit{bounded}) \textit{quasitrace} on a $C^*$-algebra $A$ is a function $\tau: A \rightarrow \mathbb{C}$ such that:

(i) $0 \leq \tau\left(x^* x\right)=\tau\left(x x^*\right)$ for all $x$ in $A$;

(ii) $\tau$ is linear on commutative ${ }^*$-subalgebras of $A$;

(iii) If $x=a+i b$ with $a, b$ self-adjoint, then
$
\tau(x)=\tau(a)+i \tau(b).
$

If $\tau$ extends to a quasitrace on $M_2(A)$, then $\tau$ is called a 2-quasitrace. A linear quasitrace is a trace.

If $A$ is unital and $\tau(1)=1$, then we say  $\tau$ is \textit{normalized}. Denote by $\QT^1_2(A)$ the space of all the normalized 2-quasitraces on $A$ and by $T(A)$ the space of all the tracial states on $A$. 
\end{definition}
\begin{remark}\label{quasi tr}
It is an open question  whether every 2-quasitrace on a $C^*$-algebra is a trace (asked by Kaplansky). A theorem of Haagerup \cite{Haagerup} says that if $A$ is exact and unital then every bounded 2-quasitrace on $A$ is a trace. This theorem can be extended to obtain that every lower semicontinuous 2-quasitrace (not necessarily bounded) on an exact $C^*$-algebra must be a trace (see \cite[Remark 2.29(i)]{BK}). Brown and Winter \cite{BW}  presented a short proof of Haagerup's result in the finite nuclear dimension case.
\end{remark}

\begin{definition}\rm \label{d def}
   Let $\tau$ be a lower semicontinuous $2$-quasitrace on $A$, extended canonically to $A\otimes\mathcal K$. We define a map $d_\tau:(A\otimes \mathcal{K})_+\rightarrow[0,\infty]$ by
  $$
  d_\tau(a)=\lim_{n\rightarrow\infty}\tau(a^{1/n}).
  $$
  It has the following properties:

  (1) if $a\precsim b$, then $d_\tau(a)\leq d_\tau(b)$;

  (2) if $a$ and $b$ are mutually orthogonal, then $d_\tau(a+b)=d_\tau(a)+ d_\tau(b)$;

  (3) $d_\tau((a-\varepsilon)_+)\rightarrow d_\tau(a)$ ($\varepsilon\rightarrow 0$).

   This map depends only on the Cuntz equivalence class of $a\in A\otimes \mathcal{K}$. Hence, we will also $d_\tau$ to denote the induced functional on ${\rm Cu}(A)$.
\end{definition}
\begin{definition}\label{def:dimension-finite}
The \emph{dimension-finite ideal} of $\tau$ is
$$
 \mathcal D_\tau(A)
 =\operatorname{span}\{a\in A_+:d_\tau(a)<\infty\}.
$$
The quasitrace $\tau$ is \emph{densely finite} if
$d_\tau(a)<\infty$ for every $a\in\Ped(A)_+$, equivalently
$\Ped(A)\subseteq\mathcal D_\tau(A)$. For lower semicontinuous
$2$-quasitraces this is equivalent to being finite on a dense ideal,
the usual definition of dense finiteness; see \cite{BK,ERS}.
It is not hard  to see $\mathcal D_\tau(A)$ is an
algebraic ideal with positive cone
$\{a\in A_+:d_\tau(a)<\infty\}$ and is contained in $N_\tau=\{a\in A: \tau(a^*a)<\infty\}$.
\end{definition}
\subsection{Algebraic ideals}

\begin{definition}\label{def:algebraic-terms}
A $C^*$-algebra $A$ is \emph{simple} if it has no nontrival closed ideals.  It is \emph{algebraically simple} if it contains no nontrival ideals. An ideal $J$ of
$A$ is \emph{semiprime} if
$$
\{xy: x,y\in I\}\subseteq J\quad\Longrightarrow\quad I\subseteq J
$$
for every algebraic two-sided ideal $I\lhd A$.
\end{definition}
A simple unital $C^*$-algebra is algebraically simple. A nonunital simple $C^*$-algebra need not be algebraically simple.There is one especially important and well-behaved ideal in any $C^*$-
algebra; see \cite[Section~5.6]{Pedersen}.

\begin{theorem} \label{thm:dense-maximal}
Let $A$ be a ${C}^*$-algebra, there is a dense ideal ${\rm Ped}(A)$ in $A$ which is contained in every dense ideal of $A$. ${\rm Ped}(A)$ is called the Pedersen ideal of $A$, and has the following properties:

{\rm (i)} it contains all projections in $A$;

{\rm (ii)} it is self-adjoint, hereditary, and spanned by its positive elements.

{\rm (iii)} it is strongly invariant:
 $x^* x \in {\rm Ped}(A) \Leftrightarrow x x^* \in {\rm Ped}(A).$ 
\end{theorem}

\begin{remark}
If $A$ is simple, let $p\in A$ be a nonzero projection. Then $p$ is  a full projection, and the ideal generate by $p$ is exactly
$\Ped(A).$
\end{remark}

\begin{definition}\label{def:pedersen-quotient}
For a nonzero $C^*$-algebra $A$, denote by $\Sigma(A)$ the algebraic quotient algebra
$
A/\Ped(A).
$
In general, $\Ped(A)$ need not be norm closed, so $\Sigma(A)$ need not be a $C^*$-algebra.
\end{definition}

\begin{theorem}\label{thm:dense-closed}
Let $A$ be a nonzero $C^*$-algebra and let $M$ be a maximal ideal.
Exactly one of the following holds.
\begin{enumerate}[label=\textup{(\roman*)},leftmargin=2.4em]
\item $M$ is dense, $\Ped(A)\subseteq M$, and
      $\Sigma(A)/(M/\Ped(A))$ is nonzero and algebraically simple;
\item $M$ is norm closed and $A/M$ is algebraically simple.
\end{enumerate}
Conversely, ideals $N\lhd\Sigma(A)$ with $\Sigma(A)/N$ nonzero and
algebraically simple lift to dense maximal ideals of $A$, and a nonzero proper closed ideal $J$ is maximal
algebraically whenever $A/J$ is algebraically simple.
\end{theorem}

\begin{proof}
If $M$ is dense, Theorem~\ref{thm:dense-maximal} gives
$\Ped(A)\subseteq M$, and the correspondence theorem gives the assertion in
the quotient.  Conversely, the inverse image of such an ideal $N$ is maximal,
contains the nonzero ideal $\Ped(A)$, and is dense.

Suppose that $M$ is not dense.  Then $\overline M$ is a proper algebraic
ideal containing $M$.  Maximality gives $M=\overline M$.  Ideals of $A/M$
lift to ideals of $A$ containing $M$, so $A/M$ is algebraically simple.  The
last assertion follows from the same ideal correspondence.
\end{proof}

\begin{corollary}
\label{cor:no-max-exact}
A nonzero $C^*$-algebra $A$ has no maximal ideals if and only if both of the
following conditions hold:
\begin{enumerate}[label=\textup{(\roman*)},leftmargin=2.4em]
\item $\Sigma(A)$ has no nonzero algebraically simple quotient.
\item there is no nonzero proper closed ideal $J\lhd A$ such that $A/J$ is
      algebraically simple.
\end{enumerate}
\end{corollary}

\begin{proposition}(\cite[Proposition 5.16]{GKT})\label{prop:max-semiprime}
Every maximal ideal of a $C^*$-algebra is prime and thus semiprime.
\end{proposition}


We use the following exact consequences of
\cite[Theorem~5.2 and Corollary~5.5]{GKT}.

\begin{proposition}\label{prop:semiprime-structure}
Let $J$ be a semiprime algebraic ideal of a $C^*$-algebra $A$.  Then:
\begin{enumerate}[label=\textup{(\roman*)},leftmargin=2.4em]
\item $J$ is idempotent, positively spanned, and self-adjoint;
\item $J$ is hereditary: $0\leq x\leq y\in J_+$ implies $x\in J$;
\item $J$ is strongly invariant:
      $x^*x\in J$ if and only if $xx^*\in J$;
\item $y\in J_+$ implies $y^{1/2}\in J$.
\end{enumerate}
\end{proposition}


\begin{remark} \label{rem:semiprime-Cu-gap}
Let $M$ be a maximal ideal and let $a,b\in A_+$ with $a\precsim b\in M$.
For every $\varepsilon>0$, there exist $\delta>0$ and $x\in A$ such that
$$
 (a-\varepsilon)_+=x^*(b-\delta)_+x.
$$
Proposition~\ref{prop:semiprime-structure} gives
$(b-\delta)_+\in M$, and hence $(a-\varepsilon)_+\in M$.  One cannot pass to
the norm limit and conclude $a\in M$. We provide a concrete example here and this issue will be addressed in
Theorem~\ref{thm:closed-face-factorization}. Let $A=\K(\ell^2)$ and $J=\bigcup_{m=1}^\infty S_m$, in which $S_m$ denotes the Schatten $m$-class. Note that $J_+$ is closed under the square-root and thus $J$ is semiprime by \cite[Theorem A]{GKT}. Now, let $a=\operatorname{diag}(2^{-n})$ and $b=\operatorname{diag}(\frac{1}{\log(n+1)})$. Then it can be verified that $a\notin J, b\in J$ and $[a]=[b]$ in $\operatorname{Cu}(A)$.
\end{remark}

\section{Generalized singular functions}
\label{sec:singular}
In this section, we first introduce some generalized comparison results, and then we construct the generalized singular functions for $C^*$-algebras and establish similar properties.

\begin{lemma}\label{lem:polar-cutdown}
For $x\in A$ and $\alpha\geq0$, there is $x_\alpha\in A$ such that
$$
 |x_\alpha|=(|x|-\alpha)_+,
 \qquad \lVert x-x_\alpha\rVert\leq\alpha.
$$
\end{lemma}

\begin{proof}
For $\alpha=0$, take $x_0=x$. For $\alpha>0$, let $x=v|x|$ be the polar decomposition in $A^{**}$ and put
$$
 g_\alpha(t)=
 \begin{cases}(t-\alpha)_+/t,&t>0,\\0,&t=0.\end{cases}
 \qquad x_\alpha=xg_\alpha(|x|).
$$
Then $x_\alpha=v(|x|-\alpha)_+\in A$, and the assertions follow from
functional calculus.
\end{proof}

\begin{lemma}\label{lem:cuntz-cutdown}
Let $x,y\in A$, let $u,v\in A^\sim$, and let
$\alpha,\beta,\delta>0$.  Then we have

{\rm (i)}  $|uxv|\precsim|x|$;

{\rm (ii)} $\|a-b\|<\eta$ $\Rightarrow$ $(|a|-\eta)_+\precsim |b|$;

{\rm (iii)} $(|x+y|-\alpha-\beta-\delta)_+
 \precsim (|x|-\alpha)_+\oplus(|y|-\beta)_+;$

{\rm (iv)} $(|uxv|-\lVert u\rVert\lVert v\rVert\alpha-\delta)_+
 \precsim (|x|-\alpha)_+.$
\end{lemma}
\label{eq:cutdown-product}\label{eq:cutdown-sum}
\begin{proof}
(i) We shall also use the basic facts $c\sim c^2$ and
 $z^*z\sim zz^*$ ($c\in A_+$, $z\in A$).

Since we have
$$
|uxv|^2=v^*x^*u^*uxv
 \leq\lVert u\rVert^2v^*x^*xv\sim v^*|x|^2v,
$$
and
$$
 v^*|x|^2v
 =(|x|v)^*(|x|v)
 \sim (|x|v)(|x|v)^*=|x|vv^*|x|
 \leq\lVert v\rVert^2|x|^2 \precsim |x|^2,
$$
Now we conclude that
$$
 |uxv|
 \sim |uxv|^2
 \precsim |x|^2
 \sim |x|.
$$

(ii) 
Let us view $A$ as a $C^*$-subalgebra of its enveloping von Neumann algebra $A^{**}$, and consider a faithful representation of $A^{**}$ as a subset of $B(\mathcal{H})$.

Denote
$$
\varepsilon = \|a-b\|,\quad z=(|a|-\eta)_{+},\quad p = \mathds{1}_{(\eta,\infty)}(|a|).$$
For every $\xi\in p\mathcal{H}$,
$$\|a\xi\| = \||a|\xi\| \ge \eta\|\xi\|.$$
Consequently,
$$\|b\xi\| \ge \|a\xi\| - \|(a-b)\xi\| \ge (\eta-\varepsilon)\|\xi\|.$$
Thus,
$$pb^{*}bp \ge (\eta-\varepsilon)^{2}p.$$
Since
$
z^{1/2}=pz^{1/2},$
 then multiplying $z^{1/2}$ on both sides gives
$$z^{1/2}b^{*}bz^{1/2} \ge (\eta-\varepsilon)^{2}z.$$
Therefore,
$$
z\precsim z^{1/2}b^{*}bz^{1/2}\sim bzb^*\leq \|z\|bb^*\sim b^*b=|b|^2\sim |b|.
$$
Finally,
$$(|a|-\eta)_{+} \precsim |b|.$$

(iii)  Let
$x_\alpha$ and $y_\beta$ be the elements supplied by
Lemma~\ref{lem:polar-cutdown}.  Thus
$$
 |x_\alpha|=(|x|-\alpha)_+,
 \qquad
 |y_\beta|=(|y|-\beta)_+,
$$
and
$$
 \lVert x-x_\alpha\rVert\leq\alpha,
 \qquad
 \lVert y-y_\beta\rVert\leq\beta.
$$
It follows that
$$
 \lVert(x+y)-(x_\alpha+y_\beta)\rVert
 \leq\alpha+\beta.
$$
Applying (ii) gives
$$
 (|x+y|-\alpha-\beta-\delta)_+
 \precsim |x_\alpha+y_\beta|.
$$
To compare the other part, applying (i) to
$$
\begin{pmatrix}
  1&1\\
  0&0
 \end{pmatrix}
 \begin{pmatrix}
  x_\alpha&0\\
  0&y_\beta
 \end{pmatrix}
 \begin{pmatrix}
  1&0\\
  1&0
 \end{pmatrix}=
  \begin{pmatrix}
  x_\alpha+ y_\beta &0\\
  0& 0
 \end{pmatrix}
$$
yields that
$$
 |x_\alpha+y_\beta|\oplus0
 \precsim
 |x_\alpha|\oplus|y_\beta|.
$$
Now we have
$$
 (|x+y|-\alpha-\beta-\delta)_+
 \precsim
 (|x|-\alpha)_+\oplus(|y|-\beta)_+.
$$

(iv) For the product estimate, if
$
0=\lVert u\rVert\lVert v\rVert.
$
 the inequality is trivial. Otherwise,
$$
 \lVert uxv-ux_\alpha v\rVert
 \leq \lVert u\rVert\lVert v\rVert\lVert x-x_\alpha\rVert
 \leq \lVert u\rVert\lVert v\rVert\alpha.
$$
Another application of (ii) gives
$$
 (|uxv|-\lVert u\rVert\lVert v\rVert\alpha-\delta)_+
 \precsim |ux_\alpha v|.
$$
Using (i) and the definition of
$x_\alpha$, we obtain
$$
 |ux_\alpha v|
 \precsim |x_\alpha|
 =(|x|-\alpha)_+.
$$
Therefore,
$$
 (|uxv|-\lVert u\rVert\lVert v\rVert\alpha-\delta)_+
 \precsim (|x|-\alpha)_+.
$$
\end{proof}

\begin{definition}
\label{eq:nu-definition}
Let $\tau$ be a lower semicontinuous $2$-quasitrace on $A$. For any fixed $x\in A$, define the $\tau$-distribution function of $x$ by
$$
 \Delta_{x,\tau}(t)=d_\tau((|x|-t)_+),
 \qquad t\geq 0.
$$
For any parameter $s>0$, define
$$
 \nu_s^\tau(x)=\inf\{t\geq0:\Delta_{x,\tau}(t)\leq s\}.
$$
The value $\nu_s^\tau(x)$ is finite because
$\Delta_{x,\tau}(\lVert x\rVert)=0$.
\end{definition}


\begin{proposition} \label{prop:nu-inequalities}
Let $\tau$ be a lower semicontinuous $2$-quasitrace on $A$.  For
$x,y\in A$, $r,s>0$,
$u,v\in A^\sim$, $\lambda\in\mathbb C$, and $a\in A_+$, one has

{\rm (i)} $\nu_{r+s}^\tau(x+y)\leq\nu_r^\tau(x)+\nu_s^\tau(y);$

 {\rm (ii)} $\nu_s^\tau(uxv)\leq\lVert u\rVert\lVert v\rVert\nu_s^\tau(x);$

{\rm (iii)} $\nu_s^\tau(\lambda x)=|\lambda|\nu_s^\tau(x);$

{\rm (iv)} $\nu_s^\tau(a^{1/2})=\nu_s^\tau(a)^{1/2}.$

Moreover, $s\mapsto\nu_s^\tau(x)$ is decreasing, and
$0\leq a\leq b$ implies $\nu_s^\tau(a)\leq\nu_s^\tau(b)$.
\end{proposition}

\begin{proof}
Choose $\alpha>\nu_r^\tau(x)$ and $\beta>\nu_s^\tau(y)$, then
$$
 d_\tau((|x|-\alpha)_+) \leq r,\quad  d_\tau((|y|-\beta)_+) \leq s,
$$
 Applying
$d_\tau$ to Lemma \ref{lem:cuntz-cutdown} (iii) gives
$$
\Delta_{x+y,\tau}(\alpha+\beta+\delta)\leq r+s.
$$
Taking infima proves (i).

For (ii),  it can be similarly deduced from
Lemma~\ref{lem:cuntz-cutdown} (iv); (iii) is trivial.

For $t>0$, since we have
$$
 (a-t^2)_+=(a^{1/2}-t)_+(a^{1/2}+t),
$$
then
$$(a^{1/2}-t)_+\sim (a-t^2)_+.$$ Hence,
$\Delta_{a^{1/2},\tau}(t)=\Delta_{a,\tau}(t^2)$, which proves
(iv).
The remaining statements follow directly from the
definition and $(a-t)_+\precsim(b-t)_+$ for $0\leq a\leq b$.
\end{proof}


\begin{proposition}\label{prop:fack-kosaki}
Let $\tau$ be a faithful, densely finite, lower semicontinuous trace on $A$, and compute generalized singular numbers in
the semifinite von Neumann algebra generated by the GNS representation of
$\tau$. Then
$$
 \nu_s^\tau(x)=\mu_s^\tau(x)
 \qquad(x\in A,\ s>0).
$$
\end{proposition}

\begin{proof}
For $t\geq0$,
$$
 d_\tau((|x|-t)_+)
   =\tau(\mathds{1}_{(t,\infty)}(|x|)).
$$
The generalized-inverse formula for $\mu_s^\tau$ therefore agrees with Definition
\ref{eq:nu-definition}; see \cite{FackKosaki}.
\end{proof}
\begin{corollary}\label{lem:threshold}
Let $b\in A_+$ and let $\tau$ be a lower semicontinuous $2$-quasitrace.  If
$L,\gamma>0$ and
$$
 \nu_s^\tau(b)\geq\gamma\qquad(0<s<L),
$$
then, for every $0\leq\eta<\gamma$,
$d_\tau((b-\eta)_+)\geq L$.
\end{corollary}


\begin{proposition}\label{prop:weighted-diagonal}
Let $(q_n)$ be pairwise orthogonal nonzero projections in $A$,
and put $p_n=\sum_{j=1}^nq_j$, $p_0=0$.
Let $(\gamma_n)$ be a decreasing sequence of positive real numbers such that $\lim\limits_{n \to \infty} \gamma_n = 0$. 
Assume $0<\tau(p_n)<\infty$. Set
$
 c=\sum_{n=1}^\infty\gamma_nq_n,
$
then
$$
 \nu_s^\tau(c)=\gamma_k, \quad\text{when}\quad \sum_{n=1}^{k-1}\tau(q_n)\leq s<\sum_{i=1}^{k}\tau(q_n).
$$
For $N\geq1 $,
$$
 \nu_s^\tau(\sum_{n=N+1}^\infty\gamma_nq_n)
 =\nu_{s+\tau(p_N)}^\tau(c).
$$
\end{proposition}

\begin{proof}
For $t>0$,
$$
 (c-t)_+=\sum_{\gamma_n>t}(\gamma_n-t)q_n,
 \qquad
 \Delta_{c,\tau}(t)=\sum_{\gamma_n>t}d_\tau(q_n)=\sum_{\gamma_n>t}\tau(q_n).
$$
If $\sum_{n=1}^{k-1}\tau(q_n)\leq s<\sum_{i=1}^{k}\tau(q_n)$, we have
$$
 \nu_s^\tau(c)=\inf\{t\geq0:\sum_{\gamma_n>t}\tau(q_n)\leq s\}=\gamma_k.
$$
The second equality is a simple translation.
\end{proof}

\section{Closed finite quasitracial-scale criterion}\label{sec:closed face}
In this section, we always assume that $A$ is a nonunital $C^*$-algebra but has  an increasing approximate unit consisting of projections $\{p_n\}$, with $p_0=0$.  We write $q_n=p_n-p_{n-1}$ and require $q_n\neq0$ for $n\geq1$. 

\begin{definition}\label{def:compact-quasitracial-data}
Let $A$ be a nonunital $C^*$-algebra with an increasing projection
approximate identity $(p_n)$, where $p_0=0$ and
$q_n=p_n-p_{n-1}\ne0$. An \textit{admissible compact quasitracial triple set} $(T, (p_n) ,\rho)$  for $A$ consists of
a family $T$, the increasing projection $(p_n)$ as an approximate identity, and a number $\rho\geq0$ satisfying the following conditions.
\begin{enumerate}[label=\textup{(\roman*)},leftmargin=3.1em]
\item $T$ is a nonempty compact convex family of densely finite,
lower semicontinuous $2$-quasitraces on $A$, normalized by
$\tau(p_1)=1$. Its topology is Hausdorff, and each function
$\tau\mapsto\tau(p_n)$, is finite and continuous.
\item For every $n,m\geq1$ and $c\in M_m(A)_+$, one has
\begin{equation*}\label{eq:compact-matricial-comparison}
 \tau(q_n)+\rho<d_{\tau^{m}}(c)\quad(\tau\in T)
 \quad\Longrightarrow\quad q_n\otimes e_{11}\precsim c.
\end{equation*}
Here $\tau^{m}$ is the matrix extension without renormalization, and
$e_{11}$ denotes the first matrix unit in $M_m$.
\end{enumerate}
\end{definition}

 Define the scale
\begin{equation*}\label{eq:finite-face-scale}
  S(\tau)=\sup_n \tau(p_n),\text{ and, }
 F=\{\tau\in T:S(\tau)<\infty\}.  
\end{equation*}

Note that $S$ is lower
semicontinuous and extended-valued affine, and $F$ is a face. We begin with the following topological property.

\begin{lemma}\label{lem:closed-finite-face}
The face $F$ is closed if and only if its finite scales have a common
finite upper bound, i.e.
\[\sup_{\tau \in F}S(\tau)<\infty.\]
The boundedness condition is vacuous when $F$ is
empty.
\end{lemma}
\begin{proof}
Suppose $\sup_{\tau\in F}S(\tau)\leq L$. Then actually $F=\{\tau\in T: S(\tau)\leq L\}$ is closed because $S$ is lower semicontinuous. Conversely, suppose
$F$ is closed but $S|_F$ is unbounded. Choose $\tau_j\in F$ with
$S(\tau_j)>4^j$. Define 
\[\sigma_N=\sum_{j=1}^N2^{-j}\tau_j+2^{-N}\tau_1,\]
which belongs to $F$ as $F$ is convex. This implies that 
\[\tau=\sum_{j\geq1}2^{-j}\tau_j=\lim_{N\to\infty}\sigma_N\in F.\] For every $j$, positivity and affinity give
\[
 S(\tau)=\sup_n \tau(p_n)\geq2^{-j}\sup_n \tau_j(p_n)>2^j,
\]
a contradiction.
\end{proof}

\begin{lemma}\label{lem:support-splitting}
Let $h\in A_+$, let $p\in A$ be a projection with $p\leq p_h$ in $A^{**}$,
and put $c=(1-p)h(1-p)$.  Then $p_c=p_h-p$ and
$$
[h]=[p]+[c]\quad\text{in }\Cu(A).
$$
Consequently, $d_\tau(h)=d_\tau(p)+d_\tau(c)$ for every lower
semicontinuous $2$-quasitrace $\tau$.
\end{lemma}

\begin{proof}
Consider the universal representation of $A$ on a Hilbert space $\mathcal{H}$, and view $A^{**}$ as a subalgebra of $B(\mathcal{H})$.

Note that
$$c = (1-p)p_hhp_h(1-p) = (p_h-p)h(p_h-p).$$
Consequently,
$$p_{c} \le p_h-p.$$
Then for any  $\xi \in (p_h-p)\mathcal{H}$, if we assume that $c\xi=0$, then
$$0 = \langle c\xi,\xi\rangle= \langle(1-p)h(1-p)\xi,\xi\rangle= \langle h\xi,\xi\rangle= ||h^{1/2}\xi||^{2}.$$
Thus,
$h^{1/2}\xi=0.$ 
Since $\ker h^{1/2} = (1-p_h)\mathcal{H}$, it follows that $\xi=0$.
Thus
$$p_{c} =p_{h}-p.$$

Since $pc = cp = 0,$  and  $p+c$ has support $p_h$. Hence  
$p+c$ and $h$ generate the same hereditary subalgebra, and hence $p+c\sim h$. Now we have $[h] = [p]+[c]$.

Applying $d_\tau$ to this equality gives
$$d_{\tau}(h) = d_{\tau}(p)+d_{\tau}(c).$$
\end{proof}

\begin{lemma}\label{lem:closed-face-growth}
Suppose that $F$ is closed in $T$. Fix $\rho\geq0$ and $R>0$ such that
$S(\tau)<R$ for every $\tau\in F$. Then there is a strictly increasing
map $\Theta:\mathbb N\to\mathbb N$, with $\Theta(n)>n$, such that
\begin{equation}\label{eq:closed-face-growth}
 \tau(p_n)\geq R\quad\Longrightarrow\quad
 \tau(p_{\Theta(n)})>2\tau(p_n)+\rho
 \qquad(\tau\in T).
\end{equation}
If $F$ is
empty, any $R>0$ may be used.
\end{lemma}
\begin{proof}
Such an $R$ exists by Lemma~\ref{lem:closed-finite-face}. Then, for each $n$, the set $K_n=\{\tau\in T:\tau(p_n)\geq R\}$ is compact
and consists of infinite-scale quasitraces. For such a $\tau\in K_n$, since $\tau(p_j)\to \infty$, there exists a $j\in \mathbb N$ such that $\tau(p_j)>2\tau(p_n)+\rho$. This implies that increasing open sets $U_j=\{\tau\in T: \tau(p_n)>2\tau(p_n)+\rho\}$ covers $K_n$ and the compactness of $K_n$ shows that there exists a $j_n$ such that
\[\tau(p_{j_n})>2\tau(p_n)+\rho.\] 
Now define $\Theta(0)=0$ and 
\[\Theta(n+1)=\max\{\Theta(n)+1, n+2, j_n\},\]
which satisfies the requirements.
\end{proof}

\begin{definition}\label{def:regular-tail}
Let $\Theta: \mathbb N \to \mathbb N$ be a strictly increasing map satisfying $\Theta(n)>n$ and fix
$N\geq0$.  Let $(\alpha_n)$ be a decreasing sequence of positive numbers tending to $0$  and let
$$
 a=\sum_{n=N+1}^\infty\alpha_{n}q_{n}.
$$
We say $a$ is a \emph{$(\Theta,N)$-regular tail} if there exists $L_{\Theta,N}>0$ such that
$$
\sup_{n\geq N+1}
 \frac{\alpha_{n}}{\alpha_{\Theta(n)}}\leq L_{\Theta,N}.
$$
\end{definition}

Thus $\alpha_n>0$ decreases to zero
and, for some $L_{\Theta,N}<\infty$,
\[
 \alpha_n\leq L_{\Theta,N}\alpha_{\Theta(n)}\qquad(n>N).
\]

\begin{theorem}\label{thm:closed-face-factorization}
Let $A$, $(p_n)$, $T$, and $\rho$ satisfy
Definition~\ref{def:compact-quasitracial-data}, and suppose that $F$
is closed. Choose $R$ as in Lemma~\ref{lem:closed-face-growth}, and
let $\Theta$ satisfy \eqref{eq:closed-face-growth}.

Let $N\geq0$, and let
\[
 a=\sum_{n=N+1}^{\infty}\alpha_nq_n
\]
be a $(\Theta,N)$-regular tail in the sense of
Definition~\ref{def:regular-tail}. 
Suppose $z\in A_+$, $C\geq1$, and $s_0\geq R$ satisfy
\begin{equation}\label{eq:closed-face-membership-hypothesis}
 \nu_s^\tau(a)\leq C\nu_s^\tau(z)
 \qquad(s\geq s_0,\ \tau\in T\setminus F).
\end{equation}
When $T=F$, this condition is vacuous.

Choose an integer $t>s_0+\rho$, and set $m=t+1$ and
$\eta=1+\|a\|+\|z\|$, and set
\[
 c=\operatorname{diag}(z,\underbrace{\eta p_1,\ldots,\eta p_1}_{t})
 \in M_m(A)_+.
\]
Then there exists $x\in M_m(A)$ such that
\begin{equation}\label{eq:closed-face-factorization}
 x^*x=a^2\otimes e_{11},\qquad
 xx^*\leq4C^2L_{\Theta,N}^2c^2
 \leq4C^2L_{\Theta,N}^2\|c\|c.
\end{equation}
Consequently, if $J$ is a semiprime algebraic ideal of $A$ containing
$z$ and $p_1$, then $a\in J$. In particular, if $J$ is a dense
semiprime ideal and $z\in J$, then $a\in J$.
\end{theorem}

\begin{proof}
Use the extensions $\tau^{(m)}$ without renormalizing, and regard $a$ as
an element of the first matrix corner $a\otimes e_{11}$. For $0<s<t$, we first show
\begin{equation}\label{eq1}
    \nu_s^{\tau^{(m)}}(c)=\inf\{\lambda\geq 0: d_{\tau^{(m)}}((c-\lambda)_+)\leq s\}=\eta\geq\|a\|\geq \nu_s^\tau(a)=\nu_s^{\tau^{(m)}}(a\otimes e_{11}).
\end{equation}

Indeed, for $\lambda<\eta$ observe that $(c-\lambda)_+=\operatorname{diag}((z-\lambda)_+, (\eta-\lambda)p_1,\dots, (\eta-\lambda)p_1))$ and therefore
\[d_{\tau^{(m)}}((c-\lambda)_+)=d_\tau((z-\lambda)_+)+t\geq t>s.\]
Then for the $\eta>\|z\|$, one has $(c-\eta)_+=0$ and thus $d_{\tau^{(m)}}((c-\eta)_+)=0<s$. Therefore, we have established \eqref{eq1}. For the case $s\geq t$, apply
\eqref{eq:closed-face-membership-hypothesis} and $z\oplus0\leq c$. Then combining the inequality \eqref{eq1},
it follows that
\begin{equation}\label{eq:closed-face-domination}
 \nu_s^{\tau^{(m)}}(a\otimes e_{11})
 \leq C\nu_s^{\tau^{(m)}}(c)
 \qquad(s>0,\ \tau\in T\setminus F).
\end{equation}
For $n\geq1$, define
\[
 \delta_n=\frac{\alpha_{\Theta(N+n)}}{2C},\qquad
 h_n=(c-\delta_n)_+.
\]
Their support projections $p_{h_n}$ increase. By definition, note that 
$\delta_n\leq \|a\|/2C< \eta$.

We claim that
\begin{equation}\label{eq:closed-face-packing-bound}
 d_{\tau^{(m)}}(h_n)>\tau(p_{N+n})-\tau(p_N)+\rho
 \qquad(\tau\in T),
\end{equation}
There are two cases.

\textbf{Case 1:} Suppose $\tau(p_{N+n})<R$. Note that $\delta_n<\eta$ implies that 
$(\eta p_1-\delta_n)_+=(\eta-\delta_n)p_1$.
and the direct calculation shows
\[
 d_{\tau^{(m)}}(h_n)\geq t>R+\rho
 >\tau(p_{N+n})-\tau(p_N)+\rho.
\]

\textbf{Case 2:}
Suppose $\tau(p_{N+n})\geq R$. Then necessarily $\tau\in T\setminus F$. This  shows that
\[\tau(p_{\Theta(N+n)})>2\tau(p_{N+n})+\rho.\]
Denote by $D=\tau(p_{\Theta(N+n)})-\tau(p_N)$ for simplicity and 
\[D>2\tau(p_{N+n})-\tau(p_N)+\rho\geq\tau(p_{N+n})+\rho\geq R+\rho>0.\]
Let $0<s<D$. Then for $a=\sum_{j>N}\alpha_jq_j$, Proposition~\ref{prop:weighted-diagonal} implies that 
\[\nu_s^\tau(a)\geq \alpha_{\Theta(N+n)}.\]
Then inequality \eqref{eq:closed-face-domination} yields
\[
 \nu_s^{\tau^{(m)}}(c)\geq \frac{\nu_s^\tau(a)}{C}\geq
 \frac{\alpha_{\Theta(N+n)}}{C}=2\delta_n.
\]
Then it follows from Corollary~\ref{lem:threshold} that $d_{\tau^{(m)}}(h_n)\geq D$ and therefore
\[
 d_{\tau^{(m)}}(h_n)\geq v_{\Theta(N+n)}(\tau)-v_N(\tau)
 >2v_{N+n}(\tau)-v_N(\tau)+\rho,
\]
by the inequality \eqref{eq:closed-face-growth} in  Lemma \ref{lem:closed-face-growth}.
This has established the inequality \eqref{eq:closed-face-packing-bound}.

We now construct mutually
orthogonal projections $r_n\sim_{\rm MvN}q_{N+n}\otimes e_{11}$ with
$f_n=\sum_{j\leq n}r_j\leq p_{h_n}$. Set $f_0=0$. Suppose $r_1,\dots, r_{n-1}$ has been constructed. At the $n$th step,
Lemma~\ref{lem:support-splitting} and
\eqref{eq:closed-face-packing-bound} give
\begin{align*}
 d_{\tau^{(m)}}((1-f_{n-1})h_n(1-f_{n-1}))
 &=d_{\tau^{(m)}}(h_n)-(v_{N+n-1}(\tau)-v_N(\tau))\\
 &>\tau(q_{N+n})+\rho\qquad(\tau\in T).
\end{align*}
Property \ref{eq:compact-matricial-comparison} in Definition \ref{def:compact-quasitracial-data} implies that $q_{N+n}\otimes e_{11}\sim_{\text{MvN}} r_n$ such that $r_n\leq p_{h_n}-f_{n-1}$.

Let $u_n\in M_m(A)$ implement these equivalences. Define $x=\sum_{n\geq1}\alpha_{N+n}u_n$, which converges
in norm and satisfies
\[
 x^*x=a^2\otimes e_{11},\qquad
 xx^*=\sum_{n\geq1}\alpha_{N+n}^2r_n.
\]
Put $\gamma_n=\alpha_{N+n}^2-\alpha_{N+n+1}^2$ and
$g(s)=\sum_{n\geq1}\gamma_n\mathds{1}_{(\delta_n,\infty)}(s)$. Therefore, one has
\[xx^*=\sum_n\gamma_nf_n\leq g(c).\]
For $s>0$,
let $n_0$ be the first index with $\delta_{n_0}<s$. Then
\[
 g(s)=\alpha_{N+n_0}^2
 \leq L_{\Theta,N}^2\alpha_{\Theta(N+n_0)}^2
 <4C^2L_{\Theta,N}^2s^2.
\]
Since $g(0)=0$, the usual functional calculus in $M_m(A)^{**}$ proves
\eqref{eq:closed-face-factorization}.

If $z,p_1\in J$, then $c\in I:= M_m(J)$. It follows from
Proposition~\ref{prop:semiprime-structure} that $J$ is idempotent. Hence,
$I$ is idempotent and semiprime by \cite[Theorem 5.2]{GKT}.
By \eqref{eq:closed-face-factorization}, one has
$a\otimes e_{11}\in I$ and thus $a\in J$.
Finally, if $J$ is dense, then necessarily $J$ contains $\Ped(A)$, and hence $p_1$ automatically,
by Theorem~\ref{thm:dense-maximal}.
\end{proof}

\begin{lemma}\label{lem:slow-variation}
Let $\Theta: \mathbb N \to \mathbb N$ be a strictly increasing map satisfying $\Theta(n)>n$. If
$M$ is  a maximal ideal of $A$, then there is  a decreasing sequence of positive numbers $(\beta_n)$ converging to $0$ such that
$$
 \beta_{\Theta(n)}\geq\frac12\beta_n
$$
for every $n\geq1$ and
$$
 b:=\sum_{n=1}^\infty\beta_nq_n\in A_+\setminus M.
$$
\end{lemma}

\begin{proof}
Propositions~\ref{prop:max-semiprime} and
\ref{prop:semiprime-structure} imply that $M$ is positively spanned.  Choose
$d\in A_+\setminus M$. Set $N_1=1$. Having chosen $N_k$, choose $N_{k+1}$ such that
$$
 N_{k+1}>\Theta(N_{k})\quad
 {\rm and}\quad
 \lVert d(1-p_{N_{k+1}-1})\rVert\leq\frac{1}{(k+1)\cdot2^{k+1}},
$$
Denote
$$
 Q_k=p_{N_{k+1}-1}-p_{N_k-1}=\sum_{n=N_k}^{N_{k+1}-1}q_n,
$$
and let
$$
 b=\sum_{k=1}^\infty\frac1kQ_k,\quad c=\sum_{k=1}^\infty k\,dQ_k
$$
Then $b,c\in A$ and $cb=d$.  Hence, $b\notin M$.

Set
$$
\beta_n=\begin{cases}
          1, & \mbox{if } 1\leq n<N_2 \\
          1/2, & \mbox{if } N_2\leq n<N_{3} \\
          1/3, & \mbox{if } N_3\leq n<N_{4} \\
          \vdots & \mbox{\vdots}
        \end{cases},
$$
        Then $$
     b=\sum_{n=1}^\infty\beta_nq_n.
        $$

For any $n\geq 1$, if $N_k\leq n<N_{k+1}$, then
$$
 n<\Theta(n)\leq\Theta(N_{k+1})<N_{k+2}.
$$
and
$$
 \beta_{\Theta(n)}\geq\frac1{k+1}\geq\frac12\beta_n.
$$
\end{proof}

\begin{theorem}\label{thm:closed-face-dense}\label{thm:closed-finite-face}
Let $A$, $(p_n)$, $T$, and $\rho$ satisfy
Definition~\ref{def:compact-quasitracial-data}. If
\[
 F=\left\{\tau\in T:\sup_n\tau(p_n)<\infty\right\}
\]
is closed, then $A$ has no maximal dense algebraic ideals. Equivalently,
every maximal ideal of $A$ is norm closed. 
\end{theorem}

\begin{proof}
By Lemma~\ref{lem:closed-finite-face}, there exists an $L$ such that $S$ is bounded by $L$ on $F$. Set $L=0$ if $F$ is empty and define $R=L+1$. Choose $\Theta$ satisfying
Lemma~\ref{lem:closed-face-growth} for the $\rho$ and $R$.

Suppose that $M$ is a maximal dense ideal of $A$. Density gives
$\Ped(A)\subseteq M$, hence $p_1\in M$, by
Theorem~\ref{thm:dense-maximal}. Moreover, $M$ is semiprime by
Proposition~\ref{prop:max-semiprime}.
Lemma~\ref{lem:slow-variation} provides a
positive contraction
\begin{equation}\label{eq:closed-face-diagonal}
 b=\sum_{n\geq1}\beta_nq_n\notin M,\qquad
 \beta_n\downarrow0,\qquad\beta_{\Theta(n)}\geq\tfrac12\beta_n.
\end{equation}
Then the maximality  of $M$  yields $z\in M$, $x_j,y_j\in A^\sim$, and $k\geq1$ with
\begin{equation}\label{eq:closed-face-maximality}
 b^{1/2}=z+\sum_{j=1}^k x_jby_j.
\end{equation}
Write $K=\sum_j\|x_j\|\|y_j\|$ and choose $\ell\geq1$ with
$2^\ell\geq k+1$.

For $\tau\in T\setminus F$ and $s\geq R$, since $\tau(p_n)\to \infty$, one chooses $n$ with
$\tau(p_{n-1})\leq s<\tau(p_n)$. Iterating
\eqref{eq:closed-face-growth} from Lemma \ref{lem:closed-face-growth} and using
Proposition~\ref{prop:weighted-diagonal}, one has
\[
 \nu_{(k+1)s}^\tau(b)\geq\beta_{\Theta^{\circ\ell}(n)}
 \geq2^{-\ell}\beta_n=2^{-\ell}\nu_s^\tau(b).
\]
Proposition~\ref{prop:nu-inequalities} applied to
\eqref{eq:closed-face-maximality} therefore yields
\begin{align*}
    2^{-l/2}\nu_s^{\tau}(b)^{1/2}\leq \nu_{(k+1)s}^{\tau}(b^{1/2})\leq \nu_s^\tau(z)+\sum_{j=1}^k\nu_s^\tau(x_jby_j)\leq \nu_s^\tau(z)+K\nu_s^\tau(b).
\end{align*}

Set $B_r=\sup_{\tau\in T}\tau(p_r)<\infty$ since $T$ is compact. Observe that if $s>B_r$ then $\nu_s^\tau(b)\leq \beta_{r+1}$ for any $\tau\in T$ by Proposition \ref{prop:weighted-diagonal}.
Since $\beta_{r+1}\to 0$, one may choose $r$ large enough such that $K\sqrt{\beta_{r+1}}\leq 2^{-(l/2)-1}$. Then set $s_0=\max\{R, B_r\}$ and thus for any $s\geq s_0$ and $\tau\in T$, one has 
\[K\nu_s^\tau(b)\leq K\sqrt{\beta_{r+1}}\cdot\sqrt{\nu_s^\tau(b)}\leq 2^{-(l+1)/2}\sqrt{\nu_s^\tau(b)}.\]
By writing $C=2^{1+l/2}$, this further implies that 
\begin{equation}\label{eq:closed-face-tail-domination}
 \nu_s^\tau(b^{1/2})\leq C\nu_s^\tau(z)=C\nu_s^\tau(|z|)
 \qquad(s\geq s_0,\ \tau\in T\setminus F).
\end{equation}
If $T=F$, take $s_0=R$ and $C=1$, since this condition is vacuous.

Now $a=b^{1/2}=\sum_n\sqrt{\beta_n}q_n$ is a $(\Theta,0)$-regular
tail with $L_{\Theta,0}\leq\sqrt2$. By
Proposition~\ref{prop:semiprime-structure}, $|z|\in M_+$.
Apply Theorem~\ref{thm:closed-face-factorization} with $J=M$, $N=0$,
and positive comparison element $|z|$. It gives $b^{1/2}\in M$,
and hence $b\in M$, contradicting \eqref{eq:closed-face-diagonal}.
The equivalent assertion about norm-closed maximal ideals follows
from Theorem~\ref{thm:dense-closed}.
\end{proof}

\begin{corollary}\label{cor:simple-quasitracial-scale}
Let $A$, $(p_n)$, $T$, and $\rho$ satisfy
Definition~\ref{def:compact-quasitracial-data}. If $A$ is simple and
\[
 F=\left\{\tau\in T:\sup_n\tau(p_n)<\infty\right\}
\]
is closed,
then $A$ has no maximal ideals.
\end{corollary}

\begin{proof}
A maximal ideal in a simple $C^*$-algebra is dense, so
Theorem~\ref{thm:closed-face-dense} applies.
\end{proof}

\begin{corollary}[Nonsimple form]\label{cor:nonsimple-quasitracial-scale}
Let $A$, $(p_n)$, $T$, and $\rho$ satisfy
Definition~\ref{def:compact-quasitracial-data}. Suppose 
\[
 F=\left\{\tau\in T:\sup_n\tau(p_n)<\infty\right\}
\]
is closed. If no
nonzero proper closed ideal
$J\lhd A$ has algebraically simple quotient $A/J$, then $A$ has
no maximal ideals.
\end{corollary}

\begin{proof}
Theorem~\ref{thm:closed-face-dense} excludes the dense alternative in
Theorem~\ref{thm:dense-closed}; the stated hypothesis excludes the closed
alternative.
\end{proof}

\section{Radius of comparison for full corners}\label{sec:radiusc}
In this section, we mainly discuss the radius of comparison for nonunital simple $C^*$-algebras. This property can be used for the verification of the condition \ref{eq:compact-matricial-comparison} in Definition \ref{def:compact-quasitracial-data}.  Moreover, it will be recorded in \ref{cor:finite-face-pedersen} that 
in many simple cases, if $T=F$, then necessarily $\Ped(A)=A$.

\begin{definition}[\cite{TomsFlat}]\label{def:strict-comparison}\label{def:comparison-radii}
Let $A$ be unital, simple, and stably finite, with
$\QT^1_2(A)\ne\varnothing$.  For $r\geq0$, we say that $A$ has
\emph{$r$-comparison} if
$$
 d_\tau(a)+r<d_\tau(b)\quad(\tau\in QT^1_2(A))
 \quad\Longrightarrow\quad a\precsim b
$$
for all $a,b\in M_\infty(A)_+$.  The \emph{radius of comparison} is
$$
 \operatorname{rc}(A)
 =\inf\{r\geq0:A\text{ has }r\text{-comparison}\}.
$$
We say that $A$ has \emph{strict comparison} if it has
$0$-comparison.  Equivalently, the strict dimension inequality implies
comparison for positive elements of $A\otimes\K$.
\end{definition}

The infimum of an empty set of comparison buffers is $\infty$.

\begin{definition}\label{def:restricted-radius}
For $A$ as above, the \emph{restricted projection radius}
$\operatorname{rc}_{\mathrm p}(A)$ is the infimum of the numbers
$r\geq0$ such that
$$
 d_\tau(q)+r<d_\tau(b)\quad(\tau\in \QT^1_2(A))
 \quad\Longrightarrow\quad q\precsim b
$$
for every projection $q\in M_\infty(A)$ and every
$b\in M_\infty(A)_+$.  Only the element on the left is required to
be a projection. It follows directly that
$$
 0\leq\operatorname{rc}_{\mathrm p}(D)
 \leq\operatorname{rc}(D)\leq\infty.
$$
\end{definition}

Using Cuntz-semigroup formulation, it follows from \cite[Section~3]{BRTTW} that the definition of radius of comparison is extended to non-unital $C^*$-algebras. Let 
$A$ be a $C^*$-algebra. We denote by $\QT_2(A)$ the collection of all extended lower semicontinuous $2$-quasitrace on $A$.  It follows from \cite[Proposition 4.2]{EHS} that the map  from $\QT_2(A)$ to $F(\Cu(A))$ defined by $\tau\mapsto d_\tau$, is bijective, in which $F(\Cu(A))$ denotes the set of all functionals on $\Cu(A)$. Thus, one obtains the following definition.

\begin{definition}(\cite[Definition 3.3.2]{BRTTW})
    Let $A$ be a $C^*$-algebra and $e\in (A\otimes \K)_+$ full. The radius of comparison of $A$ relative to $[e]$ is defined to be the infimum of the numbers $r>0$ satisfying
    \
$$ d_\tau([a])+rd_\tau([e])\leq d_\tau([b])\quad(\tau\in \QT_2(A))
 \quad\Longrightarrow\quad a\precsim b.$$
 We denote this number by $\rm rc(\Cu(A), [e])$. Similarly, one may define $\rm rc_p(\Cu(A), [e])$.
\end{definition}

\begin{remark}\label{rem:radius comparison coincide}
 It is shown in \cite[Proposition 3.2.3]{BRTTW} that for a unital $C^*$-algebra $A$ that all the quotients are stably finite (e.g., $A$ is unital, stably finite and simple), the radius of comparison $\rm rc(A)$ in the sense of Definition \ref{def:comparison-radii} is equivalent to $\rm rc(\Cu(A), [1_A])$. Similarly, one also has $\rm rc_p(A)=\rm rc_p(\Cu(A), [1_A])$.
\end{remark}

\begin{definition}
    Let $A$ be a $C^*$-algebra and $[e]\in \Cu(A)$. We denote by $\QT_2(A, [e])$ the set
    $$\QT^1_2(A, [e])=\{\tau\in \QT_2(A): d_\tau([e])=1\}.$$
\end{definition}

\begin{remark}\label{rem:Cuntz stable isomorphism}
    Let $A$ be a $\sigma$-unital simple $C^*$-algebra and $p\in A$ be a non-zero projection. For the full hereditary $C^*$-algebra $B=pAp$, it follows from the Brown stable isomorphism in \cite[Theorem 2.8]{BrownStable} that $B\otimes \K\simeq A\otimes \K$ and $\Cu(B)\simeq \Cu(A)$ via an isomorphism $\Phi$ with $\Phi([1_B]_B)=[p]_A$. Moreover, it is not hard to see $\QT^1_2(A, [p])$ can be identified by   $\QT^1_2(B)$ by the natural extension.  
\end{remark}

\begin{proposition}\label{lem:corner-radius}
    Let $A$ be a simple, stably finite and $\sigma$-unital $C^*$-algebra. Let $q\in A$ be a non-zero projection. Define $B=qAq$. Then ${\rm rc}(B)={\rm rc}(\Cu(A), [q])$ and ${\rm rc}_p(B)={\rm rc}_p(\Cu(A), [q])$. In addition, 
    for any $\rho>{\rm rc}_p(B)$, any projection $e$ and positive element $c$ in $A\otimes \K$, one has
    $$ 
 d_\tau(e)+\rho\leq d_\tau(c)\quad(\tau\in \QT^1_2(A, [q]))
 \quad\Longrightarrow\quad e\precsim c.
$$
\end{proposition}
\begin{proof}
First, it follows from Remark \ref{rem:Cuntz stable isomorphism} that $F(\Cu(A))=F(\Cu(B))$ induced by an isomorphism $\Phi: \Cu(B)\to \Cu(A)$ such that $\Phi([1_B])=[q]$. This implies that 
\[{\rm rc}(\Cu(A), [q])={\rm rc}(\Cu(B), [1_B])={\rm rc}(B)\]
and 
\[{\rm rc}_p(\Cu(A), [q])={\rm rc}_p(\Cu(B), [1_B])={\rm rc}_p(B)\]
by Remark \ref{rem:radius comparison coincide}.

For the last claim, we first shows that in the definition of $\rm rc(\Cu(A), [q])$, it suffices to look at all functionals $d_\tau\in F(\Cu(A))$ where $\tau\in \QT^1_2(A, [q])$. Indeed, for $\tau\in \QT_2(A)$, if $d_\tau([q])=\tau(q)=0$. Then for any $[a]\in \Cu(A)$, since $A$ is simple, for any $\varepsilon>0$, there exists $n\in \mathbb{N}$ such that $[(a-\varepsilon)_+]\leq n[q]$. This implies that $d_\tau([(a-\varepsilon)_+])=0$. Therefore, one has $d_\tau([a])=0$. If $d_\tau([q])=\tau(q)=\infty$, then because   $[q]$ is a compact element in $\Cu(A)$ and $A$ is simple, for any non-zero $[a]\in \Cu(A)$, there exists an $n\in \mathbb{N}$ such that $[q]\leq n[a]$. This implies that $d_\tau([a])=\infty$ for any non-zero $[a]\in \Cu(A)$. In this two cases, the inequality
$$d_\tau([a])+rd_\tau([q])\leq d_\tau([b])$$
always hold.  Then in the case $0<\tau(q)<\infty$, we define $\bar{\tau}(\cdot)=\tau(\cdot)/\tau(q)\in \QT_2^1(A, [q])$. This thus implies that the two conditions
\begin{equation}
    d_\tau([a])+rd_\tau([q])\leq d_\tau([b])\quad(\tau\in \QT_2(A))
\end{equation}
and 
\begin{equation}
    d_\tau([a])+r\leq d_\tau([b])\quad(\tau\in \QT^1_2(A, [q]))
\end{equation}
are equivalent. Thus, the final claim is established.
\end{proof}

\begin{remark}\label{rem: extreme point enough}
In the setting of Proposition \ref{lem:corner-radius}, in order to verify  $d_\tau(e)+\rho\leq d_\tau(c)$ for all $\tau\in \QT^1_2(A, [q])$, it suffices to verify the inequality for all $\tau\in \partial_e \QT_2^1(A, [q])$. This is a standard application of the Choquet simplex theory. Nevertheless, we prefer to provide a brief explanation to our $\rm rc_p(A, [q])$ case. Indeed, note that $f:\QT_2^1(A)\to (-\infty, +\infty]$ defined by $f(\tau)=d_\tau(c)-d_\tau(e)-\rho$ is lower semicontinuous and affine because $d_\tau(e)=\tau(e)$. Moreover, since $A\otimes \K$ is simple, one has $[e]\leq n[q]$ for some $n\in \mathbb{N}$. This implies that $f$ is bounded below. If f attains a negative value on $\QT_2^1(A, [q])$, then the set of all the minimizers forms a non-empty compact face whose extreme points also belong to $\partial_e \QT_2^1(A, [q])$. But this is a contradiction.
\end{remark}

\begin{proposition}\label{prop:radii-equality}
Let $A$ be a unital simple and stably finite $C^*$-algebra with $\QT_2^1(A)\neq \varnothing$ as in Definition~\ref{def:comparison-radii}. Suppose $A$ has real rank zero, or $\Cu(A)$ has weak cancellation in the sense that  $x+z\ll y+z$ implies $x\ll y$ in $\Cu(A)$ (see e.g., \cite[Definition 10.1]{GardellaPerera}). Then one has
\[ \operatorname{rc}_{\mathrm p}(A)=\operatorname{rc}(A).\]
In particular, the equality holds when $A$ has stable rank one.
\end{proposition}

\begin{proof}
It suffices to show $\operatorname{rc}(A)\leq\operatorname{rc}_{\mathrm p}(A)$.  Suppose the $\rm rc_p(A)$ is finite. Fix
$r>\operatorname{rc}_{\mathrm p}(D)$, and let $a,b\in M_n(A)_+$
satisfy $d_\tau(a)+r<d_\tau(b)$ for every normalized quasitrace.

If $A$ has real rank zero, the hereditary subalgebra generated by $a$, denoted by $(M_n(A))_a$, has an approximate identity
of projections by \cite[Theorem 2.6]{BrownPedersen}.  Given $\varepsilon>0$, choose
one of these projections $q$ with $\|a-qaq\|<\varepsilon$.
Lemma~\ref{R lem} gives
$$
 (a-\varepsilon)_+\precsim qaq\precsim q\precsim a.
$$
Thus $d_\tau(q)+r<d_\tau(b)$, and restricted comparison gives
$q\precsim b$.  It follows that $(a-\varepsilon)_+\precsim b$
for every $\varepsilon>0$, whence $a\precsim b$ by
Lemma~\ref{R lem}.

Now suppose that $\Cu(A)$ has weak cancellation. Without loss of generality, one may assume that $a$ is
contraction. Denote by $1_n$ the unit for $M_n(A)$. For $0<\varepsilon<1$, put
$$
 a_\varepsilon=(a-\varepsilon)_+,
 \qquad c=(1_n-a/\varepsilon)_+.
$$
The elements $a_\varepsilon$ and $c$ are orthogonal.  The scalar
function $t+(1-t/\varepsilon)_+$ is strictly positive on $[0,1]$,
so $a+c$ is invertible in $M_n(A)$.  Therefore
$$
 [a_\varepsilon]+[c]=[a_\varepsilon+c]\leq[1_n]
 =[a+c]\leq[a]+[c].
$$
Applying $d_\tau$ gives
$d_\tau(1_n)+r<d_\tau(b)+d_\tau(c)$, which  gives $[1_n]\leq[b]+[c]$.
Since $[1_n]\ll[1_n]$, the preceding inequalities imply
$$
 [a_\varepsilon]+[c]\ll[b]+[c].
$$
Weak cancellation yields $[a_\varepsilon]\ll[b]$, hence
$(a-\varepsilon)_+\precsim b$.  Lemma \ref{R lem}
gives $a\precsim b$.  In both cases $A$ has $r$-comparison for every
$r>\operatorname{rc}_{\mathrm p}(A)$, which proves equality of the
radii.  The stable rank one assertion follows from
\cite[Theorem~10.3]{GardellaPerera}, which is originally due to R\o rdam and Winter \cite{RW}.
\end{proof}

Define
\[
 D=p_1Ap_1,\qquad T=\QT_2^1(A,[p_1])\simeq\QT_2^1(D),
\]
and assume that $D$ is stably finite and $T\ne\varnothing$. The space $T$
is a compact Choquet simplex, with its topology inherited from the unital
corner; see e.g., \cite{BH}.

In the rest of the section, based on the radius of comparison for full corners,  we record some preliminary results showing $\Ped(A)=A$, for which 
Question \ref{q:main-question} has a ``trivial'' answer.

\begin{lemma}\label{lem:packing-pedersen}
Let $(p_n)$ be an increasing projection approximate identity of $A$, with
$p_0=0$ and $q_n=p_n-p_{n-1}$.  Suppose that, for a projection
$P\in M_m(A)$, there are mutually orthogonal projections $r_n\leq P$
equivalent to $q_n\otimes e_{11}$.  Then $\Ped(A)=A$.
\end{lemma}

\begin{proof}
Choose $u_n\in M_m(A)$ with
$u_n^*u_n=q_n\otimes e_{11}$ and $u_nu_n^*=r_n$.
Orthogonality gives $u_n^*u_m=0$ for $n\ne m$.
For $a\in A_+$ set
$$
 y_N=\Bigl(\sum_{n=1}^Nu_n\Bigr)(a^{1/2}\otimes e_{11}).
$$
If $M>N$, then
$$
 \|y_M-y_N\|^2
 =\|a^{1/2}\sum_{n=N+1}^M(q_n\otimes e_{11})a^{1/2}\|= \|a^{1/2}(p_M-p_N)a^{1/2}\|\longrightarrow0,
$$
because $(p_n)$ is an approximate identity.  Thus $y_N\to y$ in
norm for some $y\in M_m(A)$.  Since $Pu_n=Pr_nu_n=r_nu_n=u_n$, one has $Py=y$ . In addition, one has
$$
 y^*Py=\lim_Ny_N^*y_N
 =\lim_N(a^{1/2}p_Na^{1/2})\otimes e_{11}=a\otimes e_{11}.
$$
By Theorem~\ref{thm:dense-maximal}(i), $P\in\Ped(M_m(A))=M_m(\Ped(A))$
(see e.g., \cite[Section~5.6]{Pedersen}) and therefore every entry $P_{ij}$ of $P$ comes from
$\Ped(A)$.  Taking the first diagonal entry of the exact identity
above gives
$$
 a=\sum_{i,j=1}^m y_{i1}^*P_{ij}y_{j1}\in\Ped(A).
$$
All positive elements belong to $\Ped(A)$, and hence $\Ped(A)=A$.
\end{proof}

\begin{proposition}\label{cor:finite-face-pedersen}
Let $A$ be simple and nonunital, with an increasing projection approximate
identity $(p_n)$. Suppose that $D=p_1Ap_1$ is stably finite,
$\QT_2^1(D)\ne\varnothing$, and $\operatorname{rc}_{\mathrm p}(D)<\infty$.
Then $F=T$ if and only if $\Ped(A)=A$.
\end{proposition}
\begin{proof}
Suppose $F=T$ and denote by $L=\sup_{\tau\in T}S(\tau)<\infty$ by Lemma \ref{lem:closed-finite-face}. We still write $q_n=p_n-p_{n-1}$.  Fix
$\rho>\operatorname{rc}_{\mathrm p}(D)$, and choose an integer $m>L+\rho$. We now work in $B=M_m(A)$ and define $P=p_1\otimes1_m$. Let $\tau\in \QT_2^1(D)$ and denote by $\tau^m$ is natural extension on $B$. This implies that $\tau^m(P)=m$.
We now recursively construct mutually orthogonal projections $r_n\leq P$ with $r_n\sim_{\text{MvN}} q_n\otimes e_{11}$. Suppose $r_1,\dots, r_{n-1}$ has been constructed. Define $f_{n-1}=\sum_{i=1}^{n-1}r_j$ and for $P-f_{n-1}$ one has 
\[\tau^m(P-f_{n-1})-\tau^m(q_n\otimes e_{11})=m-\tau(p_n)\geq m-L>\rho.\]
Then since $\tau$ is arbitrary, it follows from Proposition~\ref{lem:corner-radius} that $q_n\otimes e_{11}\precsim_{\text{MvN}} P-f_{n-1}$, which yields $r_n\leq P-f_{n-1}$ such that $r_n\sim_{\text{MvN}} q_n\otimes e_{11}$. Then Lemma~\ref{lem:packing-pedersen} yields $\Ped(A)=A$.

For the converse, suppose $T\ne F$.  Choose $\tau\in T\setminus F$ and define
$h=\sum_n2^{-n}q_n$. 
For each $N$, note that 
\[h\geq \sum_{n=1}^N2^{-n}q_n\geq 2^{-N}p_N.\]
This implies that 
\[d_\tau(h)\geq d_\tau(p_N)=\tau(p_N)\to \infty\]
and therefore $h\notin D_\tau(A)$ (see Definition \ref{def:dimension-finite}).
On the other hand, The dimension-finite ideal $\mathcal D_\tau(A)$  contains $p_nAp_n$. This is because for any $b\in (p_nAp_n)_+$, one has
\[d_\tau(b)\leq d_\tau(p_n)=\tau(p_n)<\infty\] hence is dense and contains $\Ped(A)$. Therefore $h\notin\Ped(A)$ and thus $\Ped(A)\neq A$.
\end{proof}

The following shows that if $A$ admits a non-zero properly infinite projections, then still $\Ped(A)=A$. Recall that a nonzero projection $p$ is properly infinite if
$p\oplus p\precsim p$.

\begin{theorem}\label{thm:properly-infinite-projection}
Let $A$ be simple and $\sigma$-unital, with an approximate identity
of increasing projections.  If $A$ contains a nonzero properly infinite
projection, then  $\operatorname{Ped}(A)=A$.
\end{theorem}

\begin{proof}
The unital case follows from simplicity, so assume $A$ is nonunital. Let $(p_n)$ be the increasing projections serving as the approximate identity of $A$ and define $q_n=p_n-p_{n-1}$ as before.  Let $p\in A$ be a non-zero properly infinite projection.
Since $p$ is full and $[q_n]$ is compact,  there exists a $k_n\in \mathbb N$ such that,
\[ [q_n]\leq k_n[p]\leq[p].\]

We next construct mutually orthogonal projections $s_n\leq p$, each
equivalent to $p$.  Proper infiniteness gives orthogonal subprojections
$s_1,t_1\leq p$ with $s_1\sim_{\rm MvN}p\sim_{\rm MvN}t_1$.
The projection $t_1$ is again properly infinite.  Split it into
orthogonal $s_2,t_2$, each equivalent to $t_1$, and continue.
At step $n$, split the unused projection $t_{n-1}$; hence all the
$s_n$ are orthogonal and all are equivalent to $p$.
Since $q_n\precsim s_n$, Lemma~\ref{lem:projection-lifting} gives
$u_n\in A$ with
$$
 u_n^*u_n=q_n,\qquad r_n=u_nu_n^*\leq s_n.
$$
The $r_n$ are orthogonal subprojections of $p$.
Now, Lemma~\ref{lem:packing-pedersen}, with $m=1$ and $P=p$, implies
$\Ped(A)=A$.  
\end{proof}

We remark that if $A$ is $\sigma$-unital, simple and  purely infinite, then it follows from \cite[Theorem 2.11]{KNZpure} that any positive element is a positive combination of projections. Since $\operatorname{Ped}(A)$ contains all projections, one also has $\operatorname{Ped}(A)=A$.

\section{$C^*$-algebras without mamixal algebraic ideals.}\label{sec:stable}
In this subsection, we obtain main applications of Theorem \ref{thm:closed-face-dense} especially for $C^*$-algebras $A$ with $A\ne \Ped(A)$ . The following one is immediate.

\begin{theorem}\label{thm:finite-radius-scale}
Let $A$ be simple and nonunital, with an increasing projection approximate
identity $(p_n)$ and $p_1\ne0$. Suppose that $D=p_1Ap_1$ is stably finite,
$\QT_2^1(D)\ne\varnothing$, and $\operatorname{rc}_{\mathrm p}(D)<\infty$.
Put
\[
 T=\QT_2^1(A,[p_1])\simeq\QT_2^1(D),\qquad
 F=\left\{\tau\in T:\sup_n\tau(p_n)<\infty\right\}.
\]
If $F$ is closed, then $A$ has no maximal ideals. 
\end{theorem}

For simplicity, from now on, we denote by
$E_{\mathrm{fin}}=\partial_eT\cap F$ and
$E_\infty=\partial_eT\setminus F$.

\begin{corollary}\label{cor:finite-face-tests}
Under the comparison and projection hypotheses of
Theorem~\ref{thm:finite-radius-scale}, each of the following conditions
implies that $A$ has no maximal ideals:
\begin{enumerate}[label=\textup{(\roman*)}]
\item $A$ has no nonzero bounded lower semicontinuous $2$-quasitraces;
\item $\partial_eT$ is finite;
\item $T$ is metrizable (e.g., when $A$ is separable) and $S$ is uniformly bounded on
$E_{\mathrm{fin}}$; in particular, $E_{\mathrm{fin}}$ is finite.
\end{enumerate}

\end{corollary}
\begin{proof}
In (i), the finite scale part $F$ is empty. For (ii), if the total
extreme boundary is finite, $T=\operatorname{conv}(\partial_eT)$ and
$F=\operatorname{conv}(E_{\mathrm{fin}})$ is closed.

For (iii), under the assumption, one has
\[\tau(p_n)=\int_{\partial_eT}\sigma(p_n)d\mu_{\tau}\]
for some probability measure $\mu_\tau$ on $\partial_eT$
because $T$ is a metrizable Choquet simplex and $\tau\mapsto \tau(p_n)$ is a continuous affine function. Then the monotone convergence theorem implies that 
\[
 S(\tau)=\lim_n\int_{\partial_e T} \tau(p_n)\,d\mu_\tau
 =\int_{\partial_eT}S\,d\mu_\tau.
\]
If $S(\tau)<\infty$, then
$\mu_\tau(E_\infty)=0$ and thus $\mu_\tau(E_{\mathrm fin})=1$. Therefore, by the assumption that $S$ is uniformly bounded on $E_{\mathrm fin}$ by some $B$, one has 
\[
 S(\tau)
 =\int_{E_{\mathrm fin}}S\,d\mu_\tau\leq B
\]
holds for any $\tau\in F$. Then, Lemma~\ref{lem:closed-finite-face} yields that $F$ is closed.
\end{proof}

We now provide several concrete classes in which $C^*$-algebras have no maximal ideals. The first one consists of stable algebras.

\begin{theorem}\label{thm:stable-finite}
Let $A$ be a stable simple $C^*$-algebra with a increasing projection approximate unit $\{p_n:n\in \mathbb N\}$ with $p_1\neq 0$. Suppose $D=p_1Ap_1$ is stably finite and $\QT_2^1(D)\neq \emptyset$ and $\operatorname{rc}_{\mathrm p}(D)<\infty$. Then $A$ has no maximal ideals. In particular, all simple separable stable AF algebras have no maximal ideals.
\end{theorem}
\begin{proof}
Suppose $\tau$ is a bounded lower semicontinuous $2$-quasitrace on $A$. Let $\Phi: A\otimes \K\to A$ be an isomorphism. Then $\sigma=\tau\circ\Phi$ is non-zero bounded $2$-quasitrace on $A\otimes \K$. Then $p_n\otimes e_{11},\dots, p_n\otimes e_{mm}$ are mutually orthogonal projections with same value under $\sigma$. Define projections $Q_n=p_n\otimes\sum_{i=1}^ne_{ii}$ for all $n\in \mathbb N$, which forms an increasing approximate unit for $A\otimes \K$ and it is not hard to see $S(\sigma)=\sup_{n}\sigma(Q_n)$. On the other hand, note that 
\[S(\sigma)\geq \sigma(p_n\otimes\sum_{i=1}^me_{ii})=m\sigma(p_n\otimes e_{11})\]
holds for any $m\in \mathbb N$. This implies that $\sigma(p_n\otimes e_{ii})=0$ and thus $\sigma(Q_n)=0$. But this means $\sigma=0$. Indeed, let $a\in A\otimes \K$. First, $Q_naQ_n\leq \|a\|Q_n$ implies $\sigma(p_nap_n)\leq \|a\|\sigma(p_n)=0$. Then the lower semicontinuity of $\sigma$ implies 
\[\sigma(a)\leq \liminf_{n\to \infty}\sigma(Q_naQ_n)=0.\]
Now we apply Corollary \ref{cor:finite-face-tests}(i)
\end{proof}

\begin{corollary}\label{cor:z-stable-projectional}
Let $A$ be simple, nonunital, stably finite, and
$\mathcal Z$-stable, with an increasing projection approximate identity
$(p_n)$. If the finite-scale face of $T=\QT_2^1(A,[p_1])$ is closed,
then $A$ has no maximal ideals. In particular, the conclusion holds if
$\partial_eT$ is finite. If $A$ is separable, then the same conclusion holds if $S$ on its finite-scale extreme points $E_{\mathrm fin}$ is uniformly bounded.
\end{corollary}
\begin{proof}
The full corner $D=p_1Ap_1$ has almost unperforated Cuntz semigroup by
\cite{RordamZ}. This implies that
$\operatorname{rc}_{\mathrm p}(D)=0$; its normalized quasitrace simplex
is nonempty in this stably finite setting. Now the results follows from 
Theorem~\ref{thm:finite-radius-scale} and
Corollary~\ref{cor:finite-face-tests}. 
\end{proof}

The following one is for hereditary sub-$C^*$-algebras of certain stable $C^*$-algebras.

\begin{corollary}\label{cor:hereditary-scale-conditions}
Let $D$ be unital, simple, and stably finite, with
$\QT_2^1(D)\ne\varnothing$ and $\operatorname{rc}(D)<\infty$.
Let $A$ be a nonzero nonunital hereditary subalgebra of $D\otimes\K$
with an increasing projection approximate identity $(p_n)$. If its normalized
finite-scale face is closed, then $A$ has no maximal ideals.
In particular, if $D$ is separable, has real rank zero, and has finitely
many extreme normalized quasitraces, every nonzero hereditary subalgebra
of $D\otimes\K$ has no maximal ideals.
\end{corollary}
\begin{proof}
Write $B=p_1Ap_1$, which is also a hereditary $C^*$-algebra of $D\otimes \K$. Then the inclusion $\iota: B\to D\otimes \K$ induce an order embedding $\Phi: \Cu(B)\to \Cu(D\otimes \K)=\Cu(D)$ such that $\Phi([1_B]_B)=[p_1]_D$ (see, e.g., \cite[Lemma 5.11]{APT}). Moreover, since $B$ is full in $D\otimes \K$, then $\Phi$ has to be an isomorphism. This implies that 
$\QT_2^1(B)$ can be identified by $\QT_2^1(D\otimes \K, [p_1])$.
Then we work in $\Cu(D)$. First, by simplicity, there exists $n, m\in \mathbb N$ such that $[p_1]\leq n[1_D] $ and $[1_D] \leq m[p_1]$. This implies that for any $\tau\in \QT_2^1(D\otimes \K, [p_1])$, one has 
\[\frac{1}{n}\leq d_\tau([1_D])=\tau(1_D)\leq m.\]
In addition, note that $f:\QT_2^1(D\otimes \K, [p_1])\to \QT_2^1(D\otimes\K, [1_D])$, defined by $\tau\mapsto \tau/\tau(1_D)$, is bijective. Now, let $\rho>\rm rc(D)$ and suppose 
\[d_\tau([a])+\rho\tau(1_D)<d_\tau([b])\]
holds for any $\tau\in \QT_2^1(D\otimes \K, [p_1])=\QT_2^1(B)$. Then, this implies that 
\[d_\eta([a])+\rho<d_\eta([b])\]
holds for any $\eta\in \QT_2^1(D\otimes \K, [1_D])=\QT^1_2(D)$ and thus $a\precsim b$. This implies $\rm rc(B)\leq m\rho<\infty$.
Now apply Theorem~\ref{thm:finite-radius-scale}  to $A$.

For the last assertion, a unital hereditary subalgebra is algebraically
simple. Every nonunital hereditary subalgebra is separable and has an
increasing projection approximate identity by real rank zero
\cite{BrownPedersen}. Finally note that $f^{-1}:\QT_2^1(D)\to \QT_2^1(B)$ preserve the extreme points.
Thus, if $\QT^1_2(D)$ has finitely many extreme points, so does $\QT^1_2(B)$.
Now the result follows from Corollary~\ref{cor:finite-face-tests}(ii).
\end{proof}

\section{A negative resolution to Ozawa's question}
\label{sec:af-dense-maximal}
In this subsection, we construct a separable simple nonunital AF algebra with a dense
maximal ideal (Theorem~\ref{thm:af-dense-maximal}). This answers Ozawa's
question in the negative and shows that the closedness assumption on
$F$ in Theorem~\ref{thm:closed-face-dense} cannot be omitted.

For a subset $\mathcal S$ of a $C^*$-algebra $A$, write
\[
 \Id(\mathcal S)
 =\operatorname{span}\{uav:u,v\in A^{\sim},\ a\in\mathcal S\}
\]
for the algebraic ideal generated by $\mathcal S$, and write $\Id(a)$
when $\mathcal S=\{a\}$. All ideals below are complex-linear algebraic
two-sided ideals. Put
\[
 \Gamma=\{\beta=(\beta_n)_{n\geq1}:
 1=\beta_1\geq\beta_2\geq\cdots>0,\ \lim_n\beta_n=0\}.
\]
If $(p_n)$ is an increasing projection approximate identity, set
$p_0=0$, $q_n=p_n-p_{n-1}$, and
\[
 b_\beta=\sum_{n=1}^{\infty}\beta_nq_n
 \qquad(\beta\in\Gamma).
\]
Orthogonality of the $q_n$ implies norm convergence of this series.

\begin{lemma}\label{lem:af-diagonal-factorization}
Let $A$ have an increasing projection approximate identity $(p_n)$.
Every $a\in A$ admits a factorization $a=yb_\beta$ for some
$y\in A$ and $\beta\in\Gamma$.
\end{lemma}

\begin{proof}
Choose integers $0=N_0<N_1<N_2<\cdots$ such that
\[
 \|a(1-p_{N_k})\|\leq2^{-k}\qquad(k\geq1).
\]
Put $Q_k=p_{N_k}-p_{N_{k-1}}$ and define $\beta_n=1/k$ for
$N_{k-1}<n\leq N_k$. Then $\beta\in\Gamma$ and
$b_\beta=\sum_k k^{-1}Q_k$. The series
\[
 y=\sum_{k=1}^{\infty}k\,aQ_k
\]
converges in norm, since its $k$th summand $kaQ_k$ has norm 
$\|kaQ_k\|\leq k2^{-(k-1)}$ for $k\geq2$. Then note that the orthogonality of $Q_k$ gives
\[
 \Bigl(\sum_{k=1}^m k\,aQ_k\Bigr)b_\beta=ap_{N_m}.
\]
Taking norm limits proves $yb_\beta=a$.
\end{proof}

The next lemma gives an exact algebraic factorization from spectral
projection packing.

\begin{lemma}
\label{lem:af-exact-packing}
Let $B$ be a $C^*$-algebra, let $c\in B_+$, and let
$\beta\in\Gamma$. Suppose that $u_n\in B$ are partial isometries
whose initial projections are mutually orthogonal and whose final
projections are mutually orthogonal. Assume that
\[
 \sum_{j=1}^n u_ju_j^*
 \leq p_{(c-\sqrt{\beta_n}/4)_+}
 \qquad(n\geq1).
\]
Then there is $w\in B$ such that
\[
 w^*cw=\sum_{n=1}^{\infty}\beta_nu_n^*u_n.
\]
In particular, the element on the right belongs to $\Id(c)$.
\end{lemma}
\begin{proof}
We first write $\delta_n=\frac14\sqrt{\beta_n}$,
$P_n=p_{(c-\delta_n)_+}=1_{(\delta_n,\infty)}(c)$ and
$e_n=u_nu_n^*$ for simplicity. The projections $P_n$ increase and
$\sum_{j=1}^ne_j\leq P_n$. For $t\geq0$, define
\[
 g_n(t)=\frac{t}{\max\{t,\delta_n\}^{3/2}}.
\]
These functions are continuous, vanish at zero and satisfy
$g_n(t)=t^{-1/2}$ for $t\geq\delta_n$ and
$t g_n(t)^2\leq1$ holds for any $t$. Thus $g_n(c)\in B$. Moreover, if $n\leq L$, then $\delta_L\leq \delta_n$ holds. This implies that for $t>\delta_n$, one has $g_n(t)=g_L(t)=t^{-1/2}$. Then because $P_n=1_{(\delta_n, \infty)}(c)$, one has
\[(g_L(c)-g_n(c))P_n=(g_L(c)-g_n(c))1_{(c_n, \infty)}(c)=0.\]
Now, since $u_n=P_nu_n$, one has
\[
 (g_L(c)-g_n(c))u_n=(g_L(c)-g_n(c))P_nu_n=0.
\]
This implies $g_L(c)u_n=g_n(c)u_n$ for $n\leq L$. We now show that
$w=\sum_{n=1}^{\infty}\sqrt{\beta_n}\,g_n(c)u_n$
converges in norm. For $N\leq j$, define $f_{N,j}=\sum_{n=N}^je_n\leq P_j$.
Note that
\begin{align*}
 T_{N,L}:=\sum_{n=N}^L\beta_ne_n
 &=\sum_{j=N}^{L-1}(\beta_j-\beta_{j+1})f_{N,j}
     +\beta_Lf_{N,L}\\
 &\leq\sum_{j=N}^{L-1}(\beta_j-\beta_{j+1})
             1_{(\delta_j,\infty)}(c)
     +\beta_L1_{(\delta_L,\infty)}(c)
     =:H_{N,L}(c).
\end{align*}
For $t\geq0$, by definition the value $H_{N,L}(t)$ is zero if there is no
$j\in\{N,\ldots,L\}$ with $\delta_j<t$. Otherwise it is $\beta_j$
for the first such index. Since $\beta_j=16\delta_j^2$, one has
\[
 H_{N,L}(t)\leq\min\{\beta_N,16t^2\}
 \leq16\delta_Nt.
\]
Functional calculus in $B^{**}$ therefore gives
$T_{N,L}\leq16\delta_Nc$.

For the finite tail
$w_{N,L}=\sum_{n=N}^L\sqrt{\beta_n}\,g_n(c)u_n$,
orthogonality of the initial projections gives $u_nu_m^*=0$ for
$n\neq m$. Consequently,
\[
 w_{N,L}w_{N,L}^*
 =g_L(c)T_{N,L}g_L(c)
 \leq16\delta_N g_L(c)c g_L(c)
 \leq16\delta_N1_{B^{**}}.
\]
Thus $\|w_{N,L}\|^2\leq16\delta_N\to0$, proving norm convergence
to an element $w\in B$.

For $n\leq L$, using functional calculus, one has $g_L(t)tg_L(t)=1$ for $t>\delta_n$, which implies $g_L(c)cg_L(c)P_n=P_n$. This further entails $g_L(c)c g_L(c)u_n=u_n$ as $P_nu_n=u_n$.
Orthogonality of the final projections gives $u_n^*u_m=0$ when
$n\neq m$. In addition, recall
\[w_{1, L}=\sum_{n=1}^L\sqrt{\beta_n}g_n(c)u_n=g_L(c)\sum_{n=1}^L\sqrt{\beta_n}u_n,\] 
which implies
\[
 w_{1,L}^*cw_{1,L}=\sum_{n=1}^L\beta_nu_n^*u_n.
\]
The series on the right converges in norm by orthogonality and
$\beta_n\to0$. Taking norm limits proves the asserted identity and
gives the required algebraic factorization through $c$.
\end{proof}

We now construct the AF algebra. Let $X=\{0,1\}^{\mathbb N}$ and let
$G$ be the additive group of rational-valued locally constant functions
on $X$, equipped with
\[
 G^+=\{0\}\cup\{g\in G:g(x)>0\text{ for all }x\in X\},
 \qquad u=1.
\]
The group $G$ is countable and uniformly dense in
$C(X,\mathbb R)$. With the strict ordering above it is a simple
dimension group with state space $\operatorname{Prob}(X)$.
By the dimension-group realization theorem~\cite{EHS}, there is a
separable unital simple AF algebra $D$ with
\[
 (K_0(D),K_0(D)^+,[1_D])\cong(G,G^+,1),\qquad
 T(D)\cong\operatorname{Prob}(X).
\]
For $x=(x_i)\in X$, let $\tau_x$ be the extreme trace corresponding to
evaluation at $x$. For $n\geq1$, define
\begin{equation}\label{eq:af-scale-ranks}
 v_n(x)=2^{\Sigma_{i=1}^{n-1}x_i}+1-2^{1-n}.
\end{equation}
The exponent in $v_n(x)$ counts the coordinates equal to $1$
among $x_1,\ldots,x_{n-1}$. For $n\geq2$, we have
\[
 v_n(x)-v_{n-1}(x)
 =
 \begin{cases}
  2^{\Sigma_{i=1}^{n-2}x_i}+2^{1-n},
       &{\rm if}\,\, x_{n-1}=1,\\[1mm]
  2^{1-n},&{\rm if}\,\, x_{n-1}=0.
 \end{cases}
\]
Thus the exponential term doubles precisely when $x_{n-1}=1$.

Choose mutually orthogonal projections $q_n\in D\otimes\K$ such that
\[
 q_1=1_D\otimes e_{11},\qquad
 [q_n]_{K_0}=v_n-v_{n-1}\quad(n\geq2).
\]
Here projection classes are identified with functions in $G$ by the
displayed $K_0$-isomorphism. Each positive class has a representative
in a finite matrix algebra over $D$, and these representatives can be
placed in disjoint matrix blocks. Set
\[
 p_0=0,\qquad p_n=\sum_{j=1}^nq_j,\qquad
 A=\overline{\bigcup_{n\geq1}p_n(D\otimes\K)p_n}.
\]
Then $A$ is a nonzero hereditary subalgebra of $D\otimes\K$, so it is
simple, separable and AF. The projections $(p_n)$ form an approximate
identity. Since $q_n\neq0$ for every $n$, this approximate identity
cannot converge in norm to a unit, and $A$ is nonunital.
We identify $p_1Ap_1$ with $D$.

Continue to write $\tau_x$ for the restriction of
$\tau_x\otimes\Tr$ to $A$. These traces are densely finite and lower
semicontinuous, and
\[
 \tau_x(p_1)=1,\qquad \tau_x(p_n)=v_n(x).
\]
Since $A$ is AF, its lower semicontinuous $2$-quasitraces are traces.
Restriction to the full corner identifies
\[
 T=\QT_2^1(A,[p_1])\cong T(D)\cong\operatorname{Prob}(X).
\]
Thus $v_n$ also denotes the continuous affine function
$\tau\mapsto\tau(p_n)$ on $T$. As in Section~4, we define
\[
 S(\tau)=\sup_n\tau(p_n),\qquad
 F=\{\tau\in T:S(\tau)<\infty\},\qquad
 h=\sum_{n=1}^{\infty}2^{-n}q_n.
\]
The element $h$ is strictly positive in $A$, and
\eqref{eq:af-scale-ranks} gives
\[
 d_{\tau_x}(h)=S(\tau_x)=
 \begin{cases}
  2^{N(x)}+1,&N(x):=\sum_{i=1}^{\infty}x_i<\infty,\\
  \infty,&N(x)=\infty.
 \end{cases}
\]


\begin{proposition}\label{prop:af-finite-face-density}
The face $F$ is dense in $T$ and is not closed.
\end{proposition}
\begin{proof}
Identify $T$ with $\operatorname{Prob}(X)$. For each $m\geq1$,
define
\[
 \pi_m:X\longrightarrow X,\qquad
 \pi_m(x_1,x_2,\ldots)
 =(x_1,\ldots,x_m,0,0,\ldots).
\]
Given $\mu\in\operatorname{Prob}(X)$, put
$\mu_m=(\pi_m)_*\mu$, and let $\tau_{\mu_m}$ be the
corresponding trace in $T$.

Every point in the image of $\pi_m$ has at most $m$ coordinates
equal to $1$. Therefore, for every $n\geq1$ and $x\in X$,
\[
 v_n(\pi_m(x))
 =2^{\sum_{i=1}^{n-1}(\pi_m(x))_i}+1-2^{1-n}
 \leq 2^m+1.
\]
Consequently,
\[
 \tau_{\mu_m}(p_n)
 =\int_X v_n\,d\mu_m
 =\int_X v_n(\pi_m(x))\,d\mu(x)
 \leq2^m+1.
\]
Taking the supremum over $n$ gives
\[
 S(\tau_{\mu_m})\leq2^m+1<\infty,
\]
so $\tau_{\mu_m}\in F$.

For every $f\in C(X)$, uniform continuity gives
\[
 \|f\circ\pi_m-f\|_\infty\longrightarrow0.
\]
It follows that
\[
 \left|\int_X f\,d\mu_m-\int_X f\,d\mu\right|
 =\left|\int_X(f\circ\pi_m-f)\,d\mu\right|
 \leq\|f\circ\pi_m-f\|_\infty
 \longrightarrow0.
\]
Thus $\mu_m\to\mu$ in the weak$^*$ topology. Under the
homeomorphism $T\cong\operatorname{Prob}(X)$, this implies
$\tau_{\mu_m}\to\tau_\mu$. Since $\mu$ was arbitrary, $F$ is
dense in $T$.

To prove that $F$ is not closed, set
\[
 x^{(m)}
 =(\underbrace{1,\ldots,1}_{m},0,0,\ldots),
 \qquad
 x^{(\infty)}=(1,1,\ldots).
\]
Then $x^{(m)}\to x^{(\infty)}$ in $X$, and hence
$\tau_{x^{(m)}}\to\tau_{x^{(\infty)}}$ in $T$. On the other
hand, the formula for $S(\tau_x)$ gives
\[
 S(\tau_{x^{(m)}})=2^m+1<\infty,
 \qquad
 S(\tau_{x^{(\infty)}})=\infty.
\]
Therefore $\tau_{x^{(m)}}\in F$ for every $m$, whereas
$\tau_{x^{(\infty)}}\notin F$. Hence $F$ is not closed.
\end{proof}

\begin{definition}\label{def:af-jumps}
Let $x=(x_i)_{i\geq1}\in X$. An integer $n\geq2$ is called
a \emph{jump index} of $x$ if $x_{n-1}=1$.
For $n\geq2$, set
\[
 m_n(x)
 =
 \max\bigl(
   \{j:2\leq j<n,\ x_{j-1}=1\}\cup\{1\}
 \bigr).
\]
Thus $m_n(x)$ is the last jump index strictly less than $n$. If
there is no preceding jump before $n$, then $m_n(x)=1$.

For $\beta=(\beta_n)\in\Gamma$, define
\[
 \delta_n(\beta)=\frac14\sqrt{\beta_n},
 \qquad
 k_n(\beta)
 =
 \max\{k\geq1:\delta_n(\beta)<\frac{1}{2^k}\}.
\]
We say that a jump index $n$ of $x$ is \emph{$\beta$-good} if
\[
 m_n(x)\leq k_n(\beta).
\]
For $n\geq2$, write
\[
 E_n(\beta)
 =
 \{x\in X:x_{n-1}=1,\ m_n(x)\leq k_n(\beta)\}.
\]
\end{definition}

Since $\beta\in\Gamma$, the integers $k_n(\beta)$ are
well defined, are nondecreasing, and tend to infinity.
Each $E_n(\beta)$ depends only on the first $n-1$
coordinates of $x$, and hence is clopen. More explicitly, if
\[
 j_1<j_2<\cdots
\]
are the jump indices of $x$, with the convention $j_0=1$,
then
\[
 m_{j_r}(x)=j_{r-1},\qquad r\geq 1.
\]
Consequently,  $j_r\text{ is $\beta$-good}$ if an only if
 $
 j_{r-1}\leq k_{j_r}(\beta).
$
In particular, $j_0$ is $\beta$-good for every
$\beta\in\Gamma$.

Set
\[
 \varepsilon_n=2^{-2n-10}\qquad(n\geq2).
\]
For $\beta\in\Gamma$ and $n\geq2$, define
\[
 g_{n,\beta}(x)=
 \begin{cases}
  v_n(x)-v_{n-1}(x)-\varepsilon_n,
       &x\in E_n(\beta),\\
  \varepsilon_n,&x\notin E_n(\beta).
 \end{cases}
\]
Both $v_n$ and $v_{n-1}$ are constant and they take
rational values. Since $\varepsilon_n$ is rational and
$E_n(\beta)$ is clopen, $g_{n,\beta}$ is rational-valued
and locally constant. Thus $g_{n,\beta}\in G$.

Moreover, for every $x\in X$,
\[
 v_n(x)-v_{n-1}(x)
 =2^{\Sigma_{i=1}^{n-2}x_i}x_{n-1}+2^{1-n}
 \geq2^{1-n}>2\varepsilon_n.
\]
At each $x$, the values of $g_{n,\beta}$ and its complement
$v_n-v_{n-1}-g_{n,\beta}$ are $\varepsilon_n$ and
$v_n(x)-v_{n-1}(x)-\varepsilon_n$, in some order. Hence
\[
 g_{n,\beta}(x)\geq\varepsilon_n>0,
 \qquad
 v_n(x)-v_{n-1}(x)-g_{n,\beta}(x)
 \geq\varepsilon_n>0.
\]
Both functions therefore belong to $G^+\setminus\{0\}$. Now there exists 
a subprojection
$e_n(\beta)\leq q_n$ satisfying
$[e_n(\beta)]=g_{n,\beta}$. In particular,
\begin{equation*}\label{eq:af-beta-projection-traces}
 \tau_x(e_n(\beta))=g_{n,\beta}(x)
 \qquad(x\in X).
\end{equation*}
Set $e_1(\beta)=0$ for every $\beta\in\Gamma$.
Fix these choices for all $\beta\in\Gamma$ and $n\geq1$.

\begin{definition}
For $\beta=(\beta_n)\in\Gamma$, let 
\begin{equation*}
  a_\beta=\sum_{n=1}^{\infty}\beta_ne_n(\beta),\qquad
  e_\beta=\sum_{n=1}^{\infty}
                 \beta_n\bigl(q_n-e_n(\beta)\bigr).
\end{equation*}
Both series converge in norm and  $a_\beta, e_\beta\in A_+$.
\end{definition}
\begin{lemma} \label{lem:af-diagonal-splitting}
For every $\beta\in\Gamma$, 
\[
 a_\beta\in\Id(h),
 \qquad p_1e_\beta p_1=p_1.
\]
\end{lemma}
\begin{proof}
Since $e_1(\beta)=0$, $\beta_1=1$, and $q_1=p_1$, orthogonality also
gives $p_1e_\beta p_1=p_1$.

Now We first establish
\begin{equation*}\label{eq:af-rank-packing}
 \sum_{i=1}^{n}\tau_x(e_i(\beta))
 \leq2v_{k_n(\beta)}(x)+1,
 \qquad n\geq1,\ x\in X.
\end{equation*}
Fix $n\geq1$ and $x\in X$. If there is no $\beta$-good jump among
$2,\ldots,n$, then 
\[
 \sum_{i=1}^{n}\tau_x(e_i(\beta))
 =\sum_{i=2}^{n}\varepsilon_i<1.
\]
Otherwise, list all jump indices of $x$ in $\{2,\ldots,n\}$ as
\[
 j_1<j_2<\cdots<j_t\leq n,
\]
and set $j_0=1$. Define
\[
 r=\max\{s\in\{1,\ldots,t\}:x\in E_{j_s}(\beta)\}.
\]
Thus $j_r$ is the last $\beta$-good jump before
$n$, while $j_t$ is the last jump through $n$.

Since $j_r$ is $\beta$-good,
\[
 j_{r-1}=m_{j_r}(x)\leq k_{j_r}(\beta).
\]
It follows that $j_s\leq k_{j_r}(\beta)$ for every
$1\leq s<r$. The corresponding coordinates
$x_{j_s-1}=1$ therefore occur among
$x_1,\ldots,x_{k_{j_r}(\beta)-1}$. Consequently,
\[
 \sum_{i=1}^{k_{j_r}(\beta)-1}x_i\geq r-1,
\]
and hence
\[
 v_{k_{j_r}(\beta)}(x)
 =2^{\sum_{i=1}^{k_{j_r}(\beta)-1}x_i}
       +1-2^{1-k_{j_r}(\beta)}
 \geq2^{r-1}.
\]
Since there are exactly $r$ jumps through $j_r$, so
\[
 v_{j_r}(x)-1=2^r-2^{1-j_r}.
\]

By the definition of $r$, $x\notin E_i(\beta)$ whenever
$j_r<i\leq n$. Thus  $\tau_x(e_i(\beta))=\varepsilon_i$ for these indices.
Using $e_1(\beta)=0$ and $e_i(\beta)\leq q_i$, we obtain
\[
 \begin{aligned}
 \sum_{i=1}^{n}\tau_x(e_i(\beta))
 &\leq\sum_{i=2}^{j_r}\tau_x(q_i)
       +\sum_{i=j_r+1}^{n}\varepsilon_i\\
 &=v_{j_r}(x)-1+\sum_{i=j_r+1}^{n}\varepsilon_i\\
 &=2^r-2^{1-j_r}+\sum_{i=j_r+1}^{n}\varepsilon_i\\
 &\leq2v_{k_{j_r}(\beta)}(x)+1\\
 &\leq2v_{k_n(\beta)}(x)+1.
 \end{aligned}
\]
The last inequality uses $j_r\leq n$ and the monotonicity of
$k_n(\beta)$ and $v_n(x)$. 

Now we work in $M_5(A)$ and use the unnormalized matrix extensions
$\tau_x^{(5)}=\tau_x\otimes\Tr_5$. Set
\[
 c=\operatorname{diag}(h,h,p_1,p_1,p_1).
\]
From the the definitions of $h$ and $k_n(\beta)$, 
\[
 p_{(c-\delta_n(\beta))_+}=\operatorname{diag}
   \bigl(p_{k_n(\beta)},p_{k_n(\beta)},p_1,p_1,p_1\bigr).
\]

We inductively construct partial isometries $u_n(\beta)\in M_5(A)$
with mutually orthogonal final projections such that
\[
 u_n(\beta)^*u_n(\beta)=e_n(\beta)\otimes e_{11}
 \]
 and 
 \[
 f_n(\beta):=\sum_{i=1}^{n}u_i(\beta)u_i(\beta)^*\leq p_{(c-\delta_n(\beta))_+}.
\]
Beginning  with $u_1(\beta)=0$. Suppose that $u_1(\beta),\ldots,u_{n-1}(\beta)$ have been chosen and $$f_{n-1}(\beta)\leq  p_{(c-\delta_{n-1}(\beta))_+}.$$ 
Since
\[
 \tau_x^{(5)}
    \bigl( p_{(c-\delta_n(\beta))_+}-f_{n-1}(\beta)\bigr)
       -\tau_x(e_n(\beta))
 =2v_{k_n(\beta)}(x)+3
      -\sum_{i=1}^{n}\tau_x(e_i(\beta))
 \geq2
 \quad(x\in X).
\]
Then there exists $u_n(\beta)$ such that 
\[
 u_n(\beta)^*u_n(\beta)=e_n(\beta)\otimes e_{11},\qquad
 u_n(\beta)u_n(\beta)^*\leq p_{(c-\delta_n(\beta))_+}-f_{n-1}(\beta).
\]
 This completes the induction.

Applying Lemma~\ref{lem:af-exact-packing} in $M_5(A)$ gives $w_\beta\in M_5(A)$ such that
\[
 w_\beta^*cw_\beta
 =\sum_{n=1}^{\infty}\beta_n
       \bigl(e_n(\beta)\otimes e_{11}\bigr)
 =a_\beta\otimes e_{11}.
\]
Since $p_1=2p_1hp_1$, every entry of $c$ belongs to $\Id(h)$.
Taking the first diagonal entry of this identity proves
$a_\beta\in\Id(h)$.
\end{proof}

Now we use the decompositions
constructed above to prove the main theorem.

\begin{lemma}
\label{lem:af-finite-errors}
For every finite family $\beta^{(1)},\ldots,\beta^{(l)}\in\Gamma$,
there exists $x\in X$ such that
\[
 d_{\tau_x}(h)=\infty,
 \qquad
 d_{\tau_x}(e_{\beta^{(i)}})<3
 \quad(1\leq i\leq l).
\]
\end{lemma}
\begin{proof}
We first construct a point $x$ whose jumps are
$\beta^{(i)}$-good for every $1\leq i\leq l$.

Set $j_0=1$ and $j_1=2$. Once $j_r$ has been chosen, set
\[
 j_{r+1}
 =\min\bigl\{n>j_r:
       k_n(\beta^{(i)})\geq j_r
       \text{ for all }1\leq i\leq l\bigr\}.
\]
For each $i$, we have $k_n(\beta^{(i)})\to\infty$.
Since there are only finitely many $i$, all the displayed
inequalities hold simultaneously for sufficiently large $n$.
The minimum therefore exists at every step, and this defines
a strictly increasing sequence $j_1<j_2<\cdots$.

Now define $x=(x_m)_{m\geq1}\in X=\{0,1\}^{\mathbb N}$ by
\[
 x_m=
 \begin{cases}
  1,&m=j_r-1\text{ for some }r\geq1,\\
  0,&m\notin\{j_r-1:r\geq1\}.
 \end{cases}
\]
Thus the jump indices of $x$ are exactly $j_1,j_2,\ldots$.
For every $r\geq1$, exactly $r$ of the coordinates
$x_1,\ldots,x_{j_r-1}$ are equal to $1$. Hence
\[
 \tau_x(p_{j_r})=v_{j_r}(x)
 =2^r+1-2^{1-j_r}\longrightarrow\infty.
\]
Since $h=\sum_{n\geq1}2^{-n}q_n$, we have
$d_{\tau_x}(h)=\sup_n\tau_x(p_n)=\infty$.

We next show the estimate. Fix $i\in\{1,\ldots,l\}$.
The preceding jump before $j_r$ is $j_{r-1}$, with
the convention $j_0=1$. By construction,
\[
 m_{j_r}(x)=j_{r-1}\leq k_{j_r}(\beta^{(i)})
 \qquad(r\geq1).
\]
For $r=1$, this uses $k_2(\beta^{(i)})\geq1$;
for $r\geq2$, it is the inequality imposed when $j_r$
was chosen. Consequently, $x\in E_{j_r}(\beta^{(i)})$
for every $r\geq1$.
At $n=j_r$, we have
\[
 \tau_x\bigl(q_n-e_n(\beta^{(i)})\bigr)
 =\varepsilon_n.
\]
If $n\geq2$ is not one of the $j_r$, then $x_{n-1}=0$.
Thus $x\notin E_n(\beta^{(i)})$ and
\[
 \tau_x(q_n)=v_n(x)-v_{n-1}(x)=2^{1-n},
 \qquad
 \tau_x(e_n(\beta^{(i)}))=\varepsilon_n.
\]
In particular, both cases give
\[
 0\leq\tau_x\bigl(q_n-e_n(\beta^{(i)})\bigr)
 \leq2^{1-n}+\varepsilon_n
 \qquad(n\geq2).
\]
For $n=1$, we have $q_1-e_1(\beta^{(i)})=p_1$
and $\tau_x(p_1)=1$.

Finally, recall that
\[
 e_{\beta^{(i)}}=
 \sum_{n=1}^{\infty}
       \beta_n^{(i)}\bigl(q_n-e_n(\beta^{(i)})\bigr).
\]
The projections in this sum are mutually orthogonal.
For $t>0$, only finitely many coefficients $\beta_n^{(i)}$
exceed $t$, and functional calculus gives
\[
 (e_{\beta^{(i)}}-t)_+
 =\sum_{\{n:\,\beta_n^{(i)}>t\}}
       (\beta_n^{(i)}-t)\bigl(q_n-e_n(\beta^{(i)})\bigr).
\]
Therefore
\[
 d_{\tau_x}\bigl((e_{\beta^{(i)}}-t)_+\bigr)
 =\sum_{\{n:\,\beta_n^{(i)}>t\}}
       \tau_x\bigl(q_n-e_n(\beta^{(i)})\bigr).
\]
Every $\beta_n^{(i)}$ is strictly positive, so letting
$t$ tend to zero and using the preceding estimates yields
\[
 \begin{aligned}
 d_{\tau_x}(e_{\beta^{(i)}})
 &=\sum_{n=1}^{\infty}
       \tau_x\bigl(q_n-e_n(\beta^{(i)})\bigr)\\
 &\leq1+\sum_{n=2}^{\infty}
                 \bigl(2^{1-n}+\varepsilon_n\bigr)\\
 &=2+\sum_{n=2}^{\infty}2^{-2n-10}<3.
 \end{aligned}
\]
The point $x$ was chosen for the whole finite family,
so this estimate holds for every $1\leq i\leq l$.
\end{proof}

\begin{theorem}\label{thm:af-dense-maximal}
There exists a separable simple nonunital AF algebra
admitting a nonzero proper dense maximal algebraic ideal.
\end{theorem}

\begin{proof}
Let $A$ and $h$ be as constructed above, and set
\[
 I=\Id\{e_\beta:\beta\in\Gamma\}.
\]
We claim that $h\notin I$. Suppose that
$h\in I$, there are finitely many
$\beta^{(1)},\ldots,\beta^{(l)}\in\Gamma$ and elements
$a_i,b_i\in A^\sim$ such that
\[
 h=\sum_{i=1}^{l}a_i e_{\beta^{(i)}}b_i.
\]
Apply Lemma~\ref{lem:af-finite-errors} to this finite family.
It gives a point $x\in X$ such that
\[
 d_{\tau_x}(h)=\infty,
 \qquad
 d_{\tau_x}(e_{\beta^{(i)}})<3
 \quad(1\leq i\leq l).
\]
In particular, every $e_{\beta^{(i)}}$ belongs to the
 ideal $\mathcal D_{\tau_x}(A)$. 
Since this is an algebraic two-sided ideal, it contains
$a_i e_{\beta^{(i)}}b_i$ for every $i$, and hence, $h\in \mathcal D_{\tau_x}(A)$.

On the other hand, by Definition~\ref{def:dimension-finite}, 
\[
 \mathcal D_{\tau_x}(A)\cap A_+
 =\{a\in A_+:d_{\tau_x}(a)<\infty\}.
\]
Then $d_{\tau_x}(h)<\infty$, contradicting the choice of $x$.
Therefore $h\notin I$, and $I$ is proper.

By Lemma~\ref{lem:af-diagonal-splitting}, 
$p_1=p_1e_\beta p_1\in I$. Since $p_1\neq0$ and $A$ is simple,
$I$ is nonzero and dense.

For any $a\in A$. By Lemma~\ref{lem:af-diagonal-factorization},
there exist $y\in A$ and $\beta\in\Gamma$ with $a=yb_\beta$. Now
Lemma~\ref{lem:af-diagonal-splitting} gives
\[
 a=ya_\beta+ye_\beta\in\Id(h)+I.
\]
Then 
\begin{equation*}\label{eq:af-principal-quotient}
 A=I+\Id(h).
\end{equation*}

Consider the collection of proper algebraic ideals of $A$
containing $I$. It is nonempty because $I$ is proper.
Every ideal does not contain $h$. 
The union of a nonempty chain in this collection is still an algebraic proper
ideal containing $I$. 
Thus Zorn's lemma gives a maximal member $M$.

Any proper algebraic ideal containing $M$ also contains $I$,
and hence equals $M$ by this maximality. Therefore $M$ is a maximal
algebraic ideal of $A$. Finally,
\[
 0\neq p_1\in I\subseteq M\subsetneq A,
 \qquad
 \overline M=A,
\]
because $I$ is dense. This proves the theorem.
\end{proof}

\begin{remark}\label{rem:af-stabilization-maximal}
The dense maximal algebraic ideal in
Theorem~\ref{thm:af-dense-maximal} gives a negative answer to
Ozawa's question. Here $D=p_1Ap_1$ is AF and
$\operatorname{rc}_{\mathrm p}(D)=0$, whereas the finite-scale
face $F$ is dense and not closed. Thus the closedness assumption
in Theorem~\ref{thm:closed-face-dense} cannot be omitted.

On the other hand, $A\otimes\K$ has no maximal ideals by
Theorem~\ref{thm:stable-finite}. Consequently, the absence of
nonzero maximal algebraic ideals is not preserved by stable
isomorphism or strong Morita equivalence, even among separable
simple nonunital AF algebras.
\end{remark}

\section*{Acknowledgement}
The project began in May 2025, during the second author's visit to Dalian University of Technology, and many of the original ideas are due entirely to humans. But during  the development of this work, the authors used GPT 5.6 sol and GPT 6 astra for searching necessary references, examining mathematical arguments, and language editing.  The authors have checked and substantially revised all mathematical arguments generated by AI and take full responsibility for the final version.

\setlength{\bigskipamount}{6pt}
\end{document}